\documentclass[12pt]{article}
\usepackage{amsmath}
\usepackage{amssymb}
\usepackage{amstext}
\usepackage{amsfonts}
\usepackage{amscd}

\newtheorem{theorem}{Theorem}[section]

\newtheorem{corollary}{Corollary}[section]

\newtheorem{definition}{Definition}[section]

\newtheorem{lemma}{Lemma}[section]

\newenvironment{proof}[1][Proof]{\noindent\textbf{#1.} }{\ \rule{0.5em}{0.5em}}

\begin{document}

\begin{center}
\textbf{FUZZIFICATION OF STRONGLY COMPACT AND LOCALLY STRONGLY COMPACT SPACES}%
\end{center}

\begin{center} O. R. Sayed\\ Department of Mathematics, Faculty of Science, Assiut University,
Assiut 71516, EGYPT \\ o\_r\_sayed@yahoo.com
\end{center}

\begin{center} Adem K{\i}l{\i}\c{c}man \\ Department of Mathematics,
University Putra Malaysia, 43400 UPM, Serdang, Selangor, MALAYSIA \\ akilic@upm.edu.my \end{center}


\begin{abstract} In this paper, we study some characterizations of fuzzifying
strong compactness including nets
and pre-subbases properties. We also introduve new characterizations of locally strong compactness
in fuzzifying topology and mappings.
\end{abstract}

\section{\textbf{Introduction and Preliminaries}}

In the last few years fuzzy topology has been developed and studied by many researchers, see [7-9,
12-13, 22]. In contrast to the classical topology, fuzzy topology is endowed
with richer structure to a certain extent and generalize certain classical concepts. There are also different kind of definitions in the related revelopment for example see [11], [8], Chang [4] and Goguen [5] by using the lattices and known as \textit{L}-topological spaces. On the other hand, H$%
\ddot{o}$hle in [6] proposed a terminology of \textit{L}-fuzzy topology  and
\textit{L}-valued mapping on powerset $P(X)$, see the related works in [10, 12-13,18]. \\

In 1952, Rosser and Turquette [19] proposed an open problem: If there are many-valued theories beyond the level of predicates calculus, then what are the detail of such theories? An partial answer to this problem was given by Ying in
[23-25] by using a semantic of continuous-valued logic to
develop systematic fuzzifying topology. Roughly speaking,
the semantical analysis approach transforms the formal statements of implication formulas into logical language. In the literature there are some related references such as [1-3, 16--17, 20-21].
In particular, Ying [26] introduced the concepts of compactness and
established a generalization of Tychonoff's theorem in fuzzifying topology. Further, in [21] the concept of local compactness in fuzzifying
topology was introduced and some of properties were also established.
Despite the above development, one of the well known progress was the pre-open sets concept in [14] which were introduced by
Mashhour at el. Later in [2] the concepts of
fuzzifying pre-open sets and pre-continuity were introduced and
studied. In [3], fuzzifying pre-separation axioms and relation among these axioms were studied.

In the present study, we will use the following implication $$
[P]\leq \left[ \varphi \rightarrow \psi \right] \Leftrightarrow \ \ \left[ P
\right] \otimes \left[ \varphi \right] \leq \left[ \psi \right] .$$ 

We now give some useful definitions and results which will be used in the rest of the
present work. The family of all fuzzifying pre-open sets [2], denoted by $%
\tau _{P}\in \Im (P(X))$, and follows:
$$A\in \tau _{P}:=\forall x(x\in A\rightarrow x\in Int(Cl(A)))\ \ , {\rm i. e.}, \ \ \tau
_{P}(A)=\underset{x\in A}{\bigwedge }Int(Cl(A))(x)).$$
Similarly, the family of all fuzzifying pre-closed sets [2], denoted by $\digamma
_{P}\in \Im (P(X))$, and defined as $A\in \digamma _{P}:=X-A\in \tau _{P}$.
The fuzzifying pre-neighborhood of a point $x\in X$ [2] is denoted by
$N_{x}^{P^{X}}($or $N_{x}^{P})\in \Im (P(X))$ and defined as $N_{x}^{P}(A)=%
\underset{x\in B\subseteq A}{\bigvee }\;\tau _{P}(B).$ The fuzzifying
pre-closure of a set $A\subseteq X$ [2], denoted by $Cl_{P}\in \Im (X)$, is
defined as $Cl_{P}(A)(x)=1-N_{x}^{P}(X-A).$ \\

Note, if $(X,\tau )$ is a fuzzifying topological space and $N(X)$ is the class of all nets in $X$, then the
binary fuzzy predicates $\vartriangleright ^{P},\propto ^{P}\in \Im
(N(X)\times X)$  are defined as $$S\vartriangleright ^{P}x:=\forall
A(A\in N_{x}^{P^{X}}\rightarrow S\widetilde{\subset }A),\ \ \ S\propto
^{P}x:=\forall A(A\in N_{x}^{P^{X}}\rightarrow S\widetilde{\sqsubset }A),$$
where "$S\vartriangleright ^{P}x"$ , "$S\propto ^{P}x"$ stand for "$S$
pre-converges to $x$" , "$x$ is a pre-accumulation point of $S$",
respectively; and "$\widetilde{\subset }$", "$\widetilde{\sqsubset }$" are
the binary crisp predicates "almost in ","often in", respectively, see [20]. The
degree to which $x$ is a pre-adherence point of $S$ is $adh_{P}S(x)=[S%
\propto ^{P}x]$. If $(X,\tau )$ and $(Y,\sigma )$ are two fuzzifying
topological spaces and $f\in Y^{X}$, the unary fuzzy predicates $%
C_{P},I_{P}\in \Im (Y^{X}),$ called fuzzifying pre-continuity [2],
fuzzifying pre-irresoluteness [1], are given as $$C_{P}(f):=\forall B(B\in
\sigma \rightarrow f^{-1}(B)\in \tau _{P}),\ \ I_{P}(f):=\forall B(B\in
\sigma _{P}\rightarrow f^{-1}(B)\in \tau _{P}),$$ respectively. Let $\Omega $
be the class of all fuzzifying topological spaces. A unary fuzzy predicate $%
T_{2}^{P}\in \Im (\Omega )$, called fuzzifying pre-Hausdorffness [3], is
given as follows:\newline
$$T_{2}^{P}(X,\tau )=\forall x\forall y((x\in X\wedge y\in X\wedge x\neq
y)\rightarrow \exists B\exists C(B\in N_{x}^{P}\wedge C\in N_{y}^{P}\wedge
B\cap C\equiv \phi )).$$
A unary fuzzy predicate $\Gamma \in \Im (\Omega )$, called fuzzifying
compactness [26], is given as follows: $$\Gamma (X,\tau ):=(\forall \Re
)(K_{\circ }(~\Re ,X)\longrightarrow (\exists \wp )((\wp \leq \Re )\wedge
K(~\wp ,A)\otimes FF(\wp )))$$ and if $A\subseteq X,$ then $\Gamma
(A):=\Gamma (A,\tau /A).$ For $K$, $K_{\circ }$ (resp. $\leq $ and $FF$) see
[24, Definition 4.4] (resp. [24, Theorem 4.3] and [26, Definition 1.1 and
Lemma 1.1]). A unary fuzzy predicate $fI\in \Im (\Im (P(X))),$ called fuzzy
finite intersection property [26], is given as $$fI(\Re ):=\forall \wp ((\wp
\leq \Re )\wedge FF(\wp )\rightarrow \exists x\forall B(B\in \wp \rightarrow
x\in B)).$$ A fuzzifying topological space $(X,\tau )$ is said to be
fuzzifying $P$-topological space [26] if $\ \tau _{P}(A\cap B)\geq \tau
_{P}(A)\wedge \tau _{P}(B)$. A binary fuzzy predicate $K_{P}\in \Im (\Im
(P(X))\times P(X))$, called fuzzifying\ pre-open covering [1], is given as $%
K_{P}($~$\Re ,A):=K($~$\Re ,A)\otimes (\Re \subseteq \tau _{P})$. A unary
fuzzy predicate $\Gamma _{P}\in \Im (\Omega )$, called fuzzifying strong
compactness [1], is given as follows: $$(X,\tau )\in \Gamma _{P}:=(\forall
\Re )(K_{P}(~\Re ,X)\longrightarrow (\exists \wp )((\wp \leq \Re )\wedge
K( ~\wp ,X)\otimes FF(\wp )))$$ and if $A\subseteq X$, then $\Gamma
_{P}(A):=\Gamma _{P}(A,\tau /A)$. It is obvious that $$\Gamma _{P}(X,\tau
):=\Gamma (X,\tau _{P})\  , \ \Gamma _{P}(A,\tau /A):=\Gamma (A,\tau _{P}/A)$$
and $$ \vDash K_{\circ }(~\Re ,A)\longrightarrow K_{P}(~\Re ,A).$$
A unary fuzzy predicate $LC\in \Im (\Omega )$, called fuzzifying locally
compactness [21], is given as follows: $$(X,\tau )\in LC:=(\forall x)(\exists
B)((x\in Int(B)\otimes \Gamma (B,\tau /B)).$$

\section{Fuzzifying pre-base and pre-subbase}

\begin{definition}
Let $(X,\tau )$ be a fuzzifying topological space and $\beta _{P}\subseteq
\tau _{P}.$ Then $\beta _{P}$ is called a pre-base of $\tau _{P}$ if $\beta
_{P}$ fulfils the condition: $\vDash A\in N_{x}^{P^{X}}\rightarrow \exists
B((B\in \beta _{P})\wedge (x\in B\subseteq A)).$
\end{definition}

\begin{theorem}
$\beta _{P}$ is a pre-base of $\tau _{P}$ if and only if $\tau _{P}=\beta
_{P}^{(\cup )},$ where $$\beta _{P}^{(\cup )}(A)=\underset{\underset{\lambda
\in \Lambda }{\bigcup }B_{\lambda }=A}{\bigvee }\ \underset{\lambda \in
\Lambda }{\bigwedge }\ \beta _{P}(B_{\lambda }).$$
\end{theorem}

\begin{proof}
Suppose that $\beta _{P}$ is a pre-base of $\tau _{P}.$ If $$\underset{%
\lambda \in \Lambda }{\bigcup }B_{\lambda }=A,$$ then from Theorem 3.1 (1)
(b) in [2], $$\tau _{P}(A)=\tau _{P}\left( \underset{\lambda \in \Lambda }{%
\bigcup }B_{\lambda }\right) \geq \underset{\lambda \in \Lambda }{\bigwedge }%
\tau _{P}(B_{\lambda })\geq \underset{\lambda \in \Lambda }{\bigwedge }\beta
_{P}(B_{\lambda }).$$ Consequently, $$\tau _{P}(A)\geq \underset{\underset{%
\lambda \in \Lambda }{\bigcup }B_{\lambda }=A}{\bigvee }\ \underset{\lambda
\in \Lambda }{\bigwedge }\ \beta _{P}\left( B_{\lambda }\right) .$$ To prove
that $$\tau _{P}(A)\leq \underset{\underset{\lambda \in \Lambda }{\bigcup }%
B_{\lambda }=A}{\bigvee }\ \underset{\lambda \in \Lambda }{\bigwedge }\
\beta _{P}\left( B_{\lambda }\right) ,$$ we first prove $$\tau _{P}(A)=%
\underset{x\in A}{\bigwedge }\ \underset{x\in B\subseteq A}{\bigvee }\tau
_{P}(B).$$ (Indeed, assume $\delta _{x}=\{B:x\in B\subseteq A\}.$ Then for
any $$f\in \underset{x\in A}{\prod }\delta _{x}\ , \ \underset{x\in A}{\bigcup }%
f(x)=A,$$ and furthermore \begin{eqnarray*}  \tau _{P}(A)&=&\tau _{P}\left( \underset{x\in A}{%
\bigcup }f(x)\right) \\ &\geq& \underset{x\in A}{\bigwedge }\tau
_{P}(f(x))\\ &\geq& \underset{f\in \underset{x\in A}{\prod }\delta _{x}}{%
\bigvee }\ \underset{x\in A}{\bigwedge }\tau _{P}(f(x))\\ &=& \underset{x\in A}{%
\bigwedge }\ \underset{x\in B\subseteq A}{\bigvee }\tau _{P}(B).\end{eqnarray*} Also
$$\tau_{P}(A)\leq \underset{x\in A}{\bigwedge }\ \underset{x\in B\subseteq A}{%
\bigvee }\tau _{P}(B).$$ Therefore $$\tau _{P}(A)=\underset{x\in A}{\bigwedge }%
\ \underset{x\in B\subseteq A}{\bigvee }\tau _{P}(B)).$$
Now, since \begin{eqnarray*} N_{x}^{P^{X}}(A)&\leq& \underset{x\in B\subseteq A}{\bigvee }%
\ \ \beta _{P}(B), \\ \tau _{P}(A)&=&\underset{x\in A}{\bigwedge }\ \underset{%
x\in B\subseteq A}{\bigvee } \tau _{P}(B)=\underset{x\in A}{\bigwedge }\
N_{x}^{P^{X}}(A) \\ &\leq& \underset{x\in A}{\bigwedge }\ \underset{x\in
B\subseteq A}{\bigvee }\ \ \beta _{P}(B)=\underset{f\in \underset{x\in A}{%
\prod }\delta _{x}}{\bigvee }\ \underset{x\in A}{\bigwedge }\beta _{P}(f(x)).\end{eqnarray*}
Then $$\tau _{P}(A)\leq \underset{\underset{\lambda \in \Lambda }{\bigcup }%
B_{\lambda }=A}{\bigvee }\ \underset{\lambda \in \Lambda }{\bigwedge }\
\beta _{P}\left( B_{\lambda }\right) .$$ Therefore $$\tau _{P}(A)=\underset{%
\underset{\lambda \in \Lambda }{\bigcup }B_{\lambda }=A}{\bigvee }\ \underset%
{\lambda \in \Lambda }{\bigwedge }\ \beta _{P}\left( B_{\lambda }\right). $$

In the other side, we assume $$\tau _{P}(A)=\underset{\underset{\lambda \in
\Lambda }{\bigcup }B_{\lambda }=A}{\bigvee }\ \underset{\lambda \in \Lambda }%
{\bigwedge }\ \beta _{P}\left( B_{\lambda }\right) $$ and we will show that $%
\beta _{P}$ is a pre-base of $\tau _{P}$, i.e., for any $A\subseteq X,$ $%
N_{x}^{P^{X}}(A)\leq \underset{x\in B\subseteq A}{\bigvee }\ \ \beta
_{P}(B). $ Indeed, if $x\in B\subseteq A,$ $\underset{\lambda \in \Lambda }{%
\bigcup }B_{\lambda }=B,$ then there exists $\lambda _{\circ }\in \Lambda $
such that $x\in B_{\lambda _{\circ }}$ and $$\underset{\lambda \in \Lambda }{%
\bigwedge }\ \beta _{P}\left( B_{\lambda }\right) \leq \beta _{P}\left(
B_{\lambda _{\circ }}\right) \leq \underset{x\in B\subseteq A}{\bigvee }\ \
\beta _{P}(B).$$ Therefore $$N_{x}^{P^{X}}(A)=\underset{x\in B\subseteq A}{%
\bigvee }\ \ \tau _{P}(B)=\underset{x\in B\subseteq A}{\bigvee }\ \underset{%
\underset{\lambda \in \Lambda }{\bigcup }B_{\lambda }=A}{\bigvee }\ \underset%
{\lambda \in \Lambda }{\bigwedge }\ \beta _{P}\left( B_{\lambda }\right)
\leq \underset{x\in B\subseteq A}{\bigvee }\ \ \beta _{P}(B).$$
\end{proof}

\begin{theorem}
Let $\beta _{P}\in \Im (P(X)).$ Then $\beta _{P}$ is a pre-base for some
fuzzifying $P$-topology $\tau _{P}$ if and only if it has the following
properties:\begin{enumerate}
\item[(1)] $\beta _{P}^{(\cup )}(X)=1;$
\item[(2)] $\vDash (A\in \beta _{P})\wedge (B\in \beta _{P})\wedge (x\in A\cap
B)\rightarrow \exists C((C\in \beta _{P})\wedge (x\in C\subseteq A\cap B).$
\end{enumerate}
\end{theorem}

\begin{proof}
If $\beta _{P}$ is a pre-base for some fuzzifying $P$-topology $\tau _{P},$
then $\tau _{P}(X)=\beta _{P}^{(\cup )}(X).$ Clearly, $\beta _{P}^{(\cup
)}(X)=1.$ In addition, if $\ x\in A\cap B$, then \begin{eqnarray*} \beta _{P}(A)\wedge \beta
_{P}(B)\leq \tau _{P}(A)\wedge \tau _{P}(B) &\leq& \tau _{P}(A\cap B)\leq
 \ \ N_{x}^{P^{X}}(A\cap B) \\ &\leq& \underset{x\in C\subseteq A\cap B}{%
\bigvee }\ \ \beta _{P}(C).\end{eqnarray*} Conversely, if $\beta _{P}$ satisfies (1) and
(2), then we have $\tau _{P}$ is a fuzzifying $P$-topology. In fact, $\tau
_{P}(X)=1.$ For any $\left\{ A_{\lambda }:\lambda \in \Lambda \right\}
\subseteq P(X),$ we set $$\delta _{\lambda }=\left\{ \left\{ B_{\Phi
_{\lambda }}:\Phi _{\lambda }\in \Lambda _{\lambda }\right\} :\underset{\Phi
_{\lambda }\in \Lambda _{\lambda }}{\bigcup }B_{\Phi _{\lambda }}=A_{\lambda
}\right\} .$$ Then for any $$f\in \underset{\lambda \in \Lambda }{\prod }%
\delta _{\lambda }, \ \ \underset{\lambda \in \Lambda }{\bigcup }\ \underset{%
B_{\Phi _{\lambda }}\in f(\lambda )}{\bigcup }B_{\Phi _{\lambda }}=\underset{%
\lambda \in \Lambda }{\bigcup }A_{\lambda }.$$ Therefore
\begin{eqnarray*} \tau _{P}\left( \underset{\lambda \in \Lambda }{\bigcup }A_{\lambda
}\right) &=& \underset{\underset{\Phi \in \Lambda }{\bigcup }B_{\Phi }=\underset%
{\lambda \in \Lambda }{\bigcup }A_{\lambda }}{\bigvee }\ \underset{\Phi \in
\Lambda }{\bigwedge }\ \beta _{P}\left( B_{\Phi }\right) \\ &\geq& \underset{%
f\in \underset{\lambda \in \Lambda }{\prod }\delta _{\lambda }}{\bigvee }\
\underset{\lambda \in \Lambda }{\bigwedge }\ \underset{B_{\Phi _{\lambda
}}\in f(\lambda )}{\bigwedge }\beta _{P}(B_{\Phi _{\lambda }}) \\ &\geq& \underset{\lambda \in \Lambda }{%
\bigwedge }\ \underset{\left\{ B_{\Phi _{\lambda }}:\Phi _{\lambda }\in
\Lambda _{\lambda }\right\} \in \delta _{\lambda }}{\bigvee }\ \underset{%
\Phi _{\lambda }\in \Lambda _{\lambda }}{\bigwedge }\ \beta _{P}(B_{\Phi
_{\lambda }})=\underset{\lambda \in \Lambda }{\bigwedge }\tau _{P}\left(
A_{\lambda }\right) .\end{eqnarray*}
Finally, we need to prove that $$\tau _{P}(A\cap B)\geq \tau _{P}(A)\wedge
\tau _{P}(B).$$ If $\tau _{P}(A)>t,\tau _{P}(B)>t,$ then there exists $$%
\left\{ B_{\lambda _{1}}:\lambda _{1}\in \Lambda _{1}\right\} , \ \ \left\{
B_{\lambda _{2}}:\lambda _{2}\in \Lambda _{2}\right\} $$ such that $$\underset{%
\lambda _{1}\in \Lambda _{1}}{\bigcup }B_{\lambda _{1}}=A,\ \underset{%
\lambda _{2}\in \Lambda _{2}}{\bigcup }B_{\lambda _{2}}=B$$ and for any $$%
\lambda _{1}\in \Lambda _{1},\ \ \beta _{P}(B_{\lambda _{1}})>t,$$ for any $%
\lambda _{2}\in \Lambda _{2},$ $\ \beta _{P}(B_{\lambda _{2}})>t.$ Now, for
any $x\in A\cap B$, there exists $\lambda _{1x}\in \Lambda _{1},$ $\lambda
_{2x}\in \Lambda _{2}$ such that $x\in B_{\lambda _{1x}}\cap B_{\lambda
_{2x}}.$ From the assumption, we know that $$t<\beta _{P}(B_{\lambda
_{1x}})\wedge \beta _{P}(B_{\lambda _{2x}})\leq \underset{x\in C\subseteq
B_{\lambda _{1x}}\cap B_{\lambda _{2x}}}{\bigvee }\beta _{P}(C)$$ and
furthermore, there exists $C_{x}$ such that $$x\in C_{x}\subseteq B_{\lambda
_{1x}}\cap B_{\lambda _{2x}}\subseteq A\cap B,\ \ \beta _{P}(C_{x})>t.$$ Since
$$\underset{x\in A\cap B}{\bigcup }C_{x}=A\cap B,$$ we have $$t\leq \underset{%
x\in A\cap B}{\bigwedge }\beta _{P}(C_{x})\leq \underset{\underset{\lambda
\in \Lambda }{\bigcup }B_{\lambda }=A\cap B}{\bigvee }\ \underset{\lambda
\in \Lambda }{\bigwedge }\ \beta _{P}(B_{\lambda })=\tau _{P}(A\cap B).$$
Now, let $\tau _{P}(A)\wedge \tau _{P}(B)=k.$ For any natural number $n,$ we have $$\tau _{P}(A)>k-\frac{1}{n}, \ \ \tau _{P}(B)>k-\frac{1}{n}$$ and
so $\tau _{P}(A\cap B)\geq k-\frac{1}{n}.$ Therefore $\tau _{P}(A\cap B)\geq
k=$ $\tau _{P}(A)\wedge \tau _{P}(B).$
\end{proof}

\begin{definition}
$\varphi _{P}\in \Im (P(X))$ is called a pre-subbase of $\tau _{P}$ if $%
\varphi _{P}^{\Cap }$ is a pre-base of $\tau _{P},$ where $$\varphi
_{P}^{\Cap }(\underset{\lambda \in \Lambda }{\bigcap }B_{\lambda })=\underset%
{\underset{\lambda \in \Lambda }{\bigcap }B_{\lambda }=A}{\bigvee }\
\underset{\lambda \in \Lambda }{\bigwedge }\ \varphi _{P}(B_{\lambda })$$, $%
\left\{ B_{\lambda }:\lambda \in \Lambda \right\} \Subset P(X),$ with $%
"\Subset "$ standing for "a finite subset of".
\end{definition}

\begin{theorem}
$\varphi _{P}\in \Im (P(X))$ is a pre-subbase of some fuzzifying $P$%
-topology if and only if \ $\varphi _{P}^{(\cup )}(X)=1.$
\end{theorem}

\begin{proof}
We only will demonstrate that $\varphi _{P}^{\Cap }$ satisfies the second
condition of Theorem 2.2, and others are obvious. In fact\begin{eqnarray*}
\varphi _{P}^{\Cap }(A)\wedge \varphi _{P}^{\Cap }(B)&=& \left( \underset{%
\underset{\lambda _{1}\in \Lambda _{1}}{\bigcap }B_{\lambda _{1}}=A}{\bigvee
}\ \underset{\lambda _{1}\in \Lambda _{1}}{\bigwedge }\ \varphi
_{P}(B_{\lambda _{1}})\right) \wedge \left( \underset{\underset{\lambda
_{2}\in \Lambda _{2}}{\bigcap }B_{\lambda _{2}}=B}{\bigvee }\ \underset{%
\lambda _{2}\in \Lambda _{2}}{\bigwedge }\ \varphi _{P}(B_{\lambda
_{2}})\right) \\ &=& \underset{\underset{\lambda _{1}\in
\Lambda _{1}}{\bigcap }B_{\lambda _{1}}=A}{\bigvee }\ \underset{\underset{%
\lambda _{2}\in \Lambda _{2}}{\bigcap }B_{\lambda _{2}}=B}{\bigvee }\left( \
\underset{\lambda _{1}\in \Lambda _{1}}{\bigwedge }\ \varphi _{P}(B_{\lambda
_{1}})\right) \wedge \left( \ \underset{\lambda _{2}\in \Lambda _{2}}{%
\bigwedge }\ \varphi _{P}(B_{\lambda _{2}})\right) \\ &\leq& \underset{\underset{\lambda \in
\Lambda }{\bigcap }B_{\lambda }=A\cap B}{\bigvee }\left( \ \underset{\lambda
\in \Lambda }{\bigwedge }\ \varphi _{P}(B_{\lambda })\right) \\ &=& \varphi
_{P}^{\Cap }(A\cap B).\end{eqnarray*}
Therefore if $x\in A\cap B,$ then $$\varphi _{P}^{\Cap }(A)\wedge \varphi
_{P}^{\Cap }(B)\leq \varphi _{P}^{\Cap }(A\cap B)\leq \underset{x\in
C\subseteq A\cap B}{\bigvee }\varphi _{P}^{\Cap }(C).$$
\end{proof}

\section{Fuzzifying $P$-compact spaces}

\begin{theorem}
Let $(X,\tau )$ be a fuzzifying topological space, $\varphi _{P}$ be a
pre-subbase of $\tau _{P}$, and \begin{enumerate}
\item[($\beta _{1}$)] $:=(\forall \Re )(K_{\varphi _{P}}(\Re ,X)\rightarrow \exists \wp
((\wp \leq \Re )\wedge K(\wp ,X)\otimes FF(\wp ))),$
where $K_{\varphi _{P}}(\Re ,X):=K(\Re ,X)\otimes (\Re \subseteq \varphi
_{P});$
\item[($\beta _{2}$)] $:=(\forall S)((S$ \ is\ a\ universal\ net \ in\ $X)\rightarrow \exists
x((x\in X)\wedge (S\rhd ^{P}x));$
\item[($\beta _{3}$)] $:=(\forall S)((S\in N(X)\rightarrow (\exists T)(\exists
x)((T<S)\wedge (x\in X)\wedge (T\rhd ^{P}x)),$ \newline
where $\ "T<S"$ stands for "T is a subnet of S";
\item[($\beta _{4}$)] $:=(\forall S)((S\in N(X)\rightarrow \lnot (adh_{P}S\equiv \phi
)); $
\item[($\beta _{5}$)] $:=(\forall \Re )(\Re \in \digamma (P(X))\wedge \Re \subseteq
\digamma _{P}\otimes fI(\Re )\rightarrow \exists x\forall A(A\in \Re
\rightarrow x\in A)).$\end{enumerate}
Then $\vDash (X,\tau )\in \Gamma _{P}\leftrightarrow \beta _{i}\
,i=1,2,...,5.$

\end{theorem}

\begin{proof}
(1) Since $\varphi _{P}\subseteq \tau _{P},$ $[\Re \subseteq \varphi
_{P}]\leq \lbrack \Re \subseteq \tau _{P}]$ for any $\Re \in \Im (P(X)).$
Then $[K_{\varphi _{P}}(\Re ,X)]\leq \lbrack K_{P}(\Re ,X)]$. Therefore \ $%
\Gamma _{P}(X,\tau )\leq \lbrack \beta _{1}].$\newline
(2) $[\beta _{2}]=\bigwedge \left\{ \underset{x\in X}{\bigvee }[S\rhd
^{P}x]:S {\rm \ is \ a \ universal \ net\  in \ } X\right\} .$ \newline
(2.1) Assume $X$ is finite. We set $X=\{x_{1},...,x_{m}\}.$ For a universal net $S$ in $X$, there exists $i_{\circ }\in \{1,...,m\}$ with $S%
\widetilde{\subset }\{x_{i_{\circ }}\}.$ If not, then for any $i\in
\{1,...,m\},$ $S\widetilde{\not\subset }\{x_{i_{\circ }}\},$ $S\widetilde{%
\subset }X-\{x_{i_{\circ }}\}$ and $S\widetilde{\subset }\bigcap%
\limits_{i=1}^{m}(X-\{x_{i}\})=\phi ,$ that is a contradiction. Therefore $%
x_{i_{\circ }}\notin A$ and $N_{x_{i_{\circ }}}^{P}(A)=0$ (see[2],Theorem
4.2 (1)) provided $S\widetilde{\not\subset }A,$ and furthermore $[S\rhd
^{P}x_{i_{\circ }}]=\underset{S\widetilde{\not\subset }A}{\bigwedge }\left(
1-N_{x_{i_{\circ }}}^{P}(A)\right) =1.$ Therefore $[\beta _{2}]=1\geq
\lbrack \beta _{1}].$\newline
(2.2) In general, to prove that $[\beta _{1}]\leq \lbrack \beta _{2}]$ we
prove that for any $\lambda \in \lbrack 0,1],$ if $[\beta _{2}]<\lambda ,$
then $[\beta _{1}]<\lambda $. Assume for any $\lambda \in \lbrack 0,1]$, $%
[\beta _{2}]<\lambda .$ Then there exists a universal net $S$ in $X$ such
that $\underset{x\in X}{\bigvee }[S\rhd ^{P}x]<\lambda $ and for any $x\in
X, $ $[S\rhd ^{P}x]=\underset{S\widetilde{\not\subset }A}{\bigwedge }\left(
1-N_{x}^{P}(A)\right) <\lambda ,$ i.e., there exists $A\subseteq X$ with $S%
\widetilde{\not\subset }A$ and $N_{x}^{P}(A)>1-\lambda .$ Since $\varphi
_{P} $ is a pre-subbase of $\tau _{P},\varphi _{P}^{\Cap }$ is a pre-base of
$\tau _{P}$ and from Definition 2.1, we have $\underset{x\in B\subseteq A}{%
\bigvee }\varphi _{P}^{\Cap }(B)\geq N_{x}^{P}(A)>1-\lambda ,$ i.e., there
exists $B\subseteq A$ such that $x\in B\subseteq A$ and $\bigvee \left\{
\underset{\lambda \in \Lambda }{\min }\varphi _{P}(B_{\lambda }):\underset{%
\lambda \in \Lambda }{\bigcap }B_{\lambda }=B,B_{\lambda }\subseteq
X,\lambda \in \Lambda \right\} =\varphi _{P}^{\Cap }(B)>1-\lambda ,$ where $%
\Lambda $ is finite. Therefore there exists a finite set $\Lambda $ and \ $%
B_{\lambda }\subseteq X(\lambda \in \Lambda )$ such that $\underset{\lambda
\in \Lambda }{\bigcap }B_{\lambda }=B$ and for any $\lambda \in \Lambda
,\varphi _{P}(B_{\lambda })>1-\lambda .$ Since $S\widetilde{\not\subset }A$
and $\Lambda $ is finite, there exists $\lambda (x)\in \Lambda $ such that $S%
\widetilde{\not\subset }B_{\lambda (x)}.$ We set $$\Re _{\circ }(B_{\lambda
(x)})=\underset{x\in X}{\bigvee }\varphi _{P}(B_{\lambda (x)}).$$ If $\wp \leq
\Re _{\circ },$ then for any $\delta >0,\wp _{\delta }\subseteq \{B_{\lambda
(x)}:x\in X\}.$ Consequently, for any $B\in \wp _{\delta },S\widetilde{%
\not\subset }B$ and $S\widetilde{\subset }B^{c}$ since $S$ is a universal
net. If \ $$[FF(\wp )]=1-\inf \left\{ \delta \in \lbrack 0,1]:F(\wp _{\delta
})\right\} =t,$$ then for any $n\in w$ (the non-negative integer), $$\inf
\left\{ \delta \in \lbrack 0,1]:F(\wp _{\delta })\right\} <1-t+\frac{1}{n},$$
and there exists $\delta _{\circ }<1-t+\frac{1}{n}$ such that $F(\wp
_{\delta \circ }).$ If $\delta _{\circ }=0,$then $P(X)=\wp _{\delta \circ }$
is finite and it is proved in (2.1). If $\delta _{\circ }>0,$ then for any $%
B\in \wp _{\delta \circ },S\widetilde{\subset }B^{c}.$ Since $F(\wp _{\delta
\circ }),$ we have $$S\widetilde{\subset }\bigcap \{B^{c}:B\in \wp _{\delta
\circ }\}\neq \phi .$$ i.e., $\bigcup \wp _{\delta \circ }\neq X$ and there
exist $x_{\circ }\in X$ such that for any $B\in \wp _{\delta \circ
},x_{\circ }\notin B.$ Therefore, if $x_{\circ }\in B,$ then $B\notin \wp
_{\delta \circ }$, i.e., $$\wp (B)<\delta \circ , K(\wp ,X)=\underset{x\in X%
}{\bigwedge }\ \underset{x\in B}{\bigvee }\wp (B)\leq \underset{x_{\circ
}\in B}{\bigvee }\wp (B)\leq \delta \circ <1-t+\frac{1}{n}.$$ Let $%
n\rightarrow \infty .$ We obtain $K(\wp ,X)\leq 1-t$ and $[K(\wp ,X)\otimes
FF(\wp )]=0.$ In addition, $[K_{\varphi _{P}}(\Re _{\circ },X)]\geq
1-\lambda .$ In fact, $[\Re _{\circ }\subseteq \varphi _{P}]=1$ and \ $$%
[K(\Re _{\circ },X)]=\underset{x\in X}{\bigwedge }\ \underset{x\in B}{%
\bigvee }\Re _{\circ }(B)\geq \underset{x\in X}{\bigwedge }\Re _{\circ
}(B_{\lambda (x)})\geq \underset{x\in X}{\bigwedge }\varphi _{P}(B_{\lambda
(x)})\geq 1-\lambda $$ where $x\in B_{\lambda (x)}.$ Now, we have
\begin{eqnarray*} [\beta _{1}]&=&(\forall \Re )(K_{\varphi _{P}}(\Re ,X)\rightarrow \exists \wp
((\wp \leq \Re )\wedge K(\wp ,X)\otimes FF(\wp )))\\ &\leq& K_{\varphi _{P}}(\Re _{\circ },X)\rightarrow \exists \wp ((\wp
\leq \Re _{\circ })\wedge K(\wp ,X)\otimes FF(\wp ))\\ &=& \min (1,1-K_{\varphi _{P}}(\Re _{\circ },X)+\underset{\wp \leq \Re
_{\circ }}{\bigvee }[K(\wp ,X)\otimes FF(\wp ))] \leq \lambda .\end{eqnarray*}
By noticing that $\lambda $ is arbitrary, we have $[\beta _{1}]\leq \lbrack
\beta _{2}].$\newline
(3) It is immediate that $[\beta _{2}]\leq \lbrack \beta _{3}].$ \newline
(4) To prove that $[\beta _{3}]\leq \lbrack \beta _{4}],$ first we prove
that $$[\exists T\ ((T<S)\wedge (T\rhd ^{P}x))]\leq [S\propto ^{P}x],$$
where $$[\exists T\ ((T<S)\wedge (T\rhd ^{P}x))]= \underset{T<S}{\bigvee }\
\underset{T\widetilde{\not\subset }A}{\bigwedge }\left(
1-N_{x}^{P}(A)\right) $$ and $$[S\propto ^{P}x]=\underset{S\widetilde{%
\not\sqsubset }A}{\bigwedge }\left( 1-N_{x}^{P}(A)\right) .$$ Indeed, for any
$T<S$ one can deduce $\{A:S\widetilde{\not\sqsubset }A\}\subseteq \{A:T%
\widetilde{\not\subset }A\}$ as follows. Suppose $T=S\circ K$ . If $S%
\widetilde{\not\sqsubset }A,$ then there exists $m\in D$ such that $%
S(n)\notin A$ when $n\geq m,$ where $\geq $ directs the domain $D$ of $S$.
Now, we will show that $T\widetilde{\not\subset }A.$ If not, then there
exists $p\in E$ such that $T(q)\in A$ when $q\geq p,$ where $\geq $ directs
the domain $E$ of $T$. Moreover, let $n_{1}\in E$ such that $%
K(n_{1})\geq m$ since $T<S,$ and there exists $n_{2}\in E$ such that $%
n_{2}\geq n_{1},p$ \ since $(E,\geq )$ directed. So, $K(n_{2})\geq
K(n_{1})\geq m,$ $S\circ K(n_{2})\notin A$ and $S\circ K(n_{2})=T(n_{2})\in
A.$ They are contrary. Hence $$ \{A:S\widetilde{\not\sqsubset }A\}\subseteq
\{A:T\widetilde{\not\subset }A\}.$$ Therefore \begin{eqnarray*} [\exists T\ ((T<S)\wedge
(T\rhd ^{P}x))]&=& \underset{T<S}{\bigvee }\ \underset{T\widetilde{%
\not\subset }A}{\bigwedge }\left( 1-N_{x}^{P}(A)\right) \\ &=& \underset{T<S}{%
\bigvee }\ \underset{\{A:T\widetilde{\not\subset }A\}}{\bigwedge }\left(
1-N_{x}^{P}(A)\right) \\ &\leq& \underset{\{A:S\widetilde{\not\sqsubset }A\}}{%
\bigwedge }\left( 1-N_{x}^{P}(A)\right) \\ &=& \underset{S\widetilde{%
\not\sqsubset }A}{\bigwedge }\left( 1-N_{x}^{P}(A)\right) = [S\propto
^{P}x].\end{eqnarray*} Therefore for any $x\in X$ and $S\in N(X)$ we have \newline
\begin{eqnarray*}[\beta _{3}]&=&\underset{S\in N(X)}{\bigwedge }\ \underset{x\in X}{\bigvee }%
[\exists T\ ((T<S)\wedge (T\rhd ^{P}x))] \\ &\leq& \underset{S\in N(X)}{%
\bigwedge }\ \underset{x\in X}{\bigvee }[S\propto ^{P}x] = \underset{S\in
N(X)}{\bigwedge }\ \lnot \left( \underset{x\in X}{\bigwedge }\left(
1-[S\propto ^{P}x]\right) \right) \\ &=& \underset{S\in N(X)}{\bigwedge }[\lnot (adh_{P}S\equiv \phi )]= [\beta
_{4}].\end{eqnarray*}
(5) We want to show that $[\beta _{4}]\leq \lbrack \beta _{5}].$ For any $%
\Re \in F(P(X)),$ assume $[fI(\Re )]=\lambda .$ Then for any $\delta
>1-\lambda ,$ if $A_{1},...,A_{n}\in \Re _{\delta },$ $A_{1}\cap A_{2}\cap
...\cap A_{n}\neq \phi .$ In fact, we set $\wp (A_{i})=\bigvee_{i=1}^{n}\Re
(A_{i}).$ Then $\wp \leq \Re $ and $FF(\wp )=1.$ By putting $\varepsilon
=\lambda +\delta -1>0,$ we obtain \newline
$$\lambda -\varepsilon <\lambda \leq [FF(\wp )\rightarrow (\exists
x)(\forall B)(B\in \wp \rightarrow x\in B)]=\underset{x\in X}{\bigvee }\
\underset{x\notin B}{\bigwedge }(1-\wp (B)).$$ There exists $x_{\circ }\in X$
such that $\lambda -\varepsilon <\underset{x_{\circ }\notin B}{\bigwedge }%
(1-\wp (B)),$ $x_{\circ }\notin B$ implies $\wp (B)<1-\lambda +\varepsilon
=\delta $ and $x_{\circ }\in \cap \wp _{\delta }=A_{1}\cap A_{2}\cap ...\cap
A_{n}.$ Now, we set $\vartheta _{\delta }=\left\{ A_{1}\cap A_{2}\cap
...\cap A_{n}:n\in N,A_{1},...,A_{n}\in \Re _{\delta }\right\} $ and $%
S:\vartheta _{\delta }\rightarrow X,B\mapsto x_{B}\in B,B\in \vartheta
_{\delta }$ and know that $(\vartheta _{\delta },\subseteq )$ is a directed
set and $S$ is a net in $X$. Therefore $$[\beta _{4}]\leq \lbrack \lnot
(adh_{P}S\equiv \phi )]=\underset{x\in X}{\bigvee }\ \underset{S\overset{%
\sim }{\not\sqsubset }A}{\bigwedge }(1-N_{x}^{P}(A)).$$ Assume $[\Re
\subseteq F_{P}]=\mu .$ Then for any $B\in P(X),\Re (B)\leq 1+F_{P}(B)-\mu ,$
and $$[\Re \subseteq F_{P}\otimes fI(A)\rightarrow (\exists x)(\forall
A)((A\in \Re )\rightarrow x\in A)]=\min (1,2-\mu -\lambda +\underset{x\in
X}{\bigvee }\ \underset{x\notin A}{\bigwedge }(1-\Re (A))).$$ Therefore it
suffices to show that for any $$x\in X,\underset{S\overset{\sim }{%
\not\sqsubset }A}{\bigwedge }(1-N_{x}^{P}(A))\leq 2-\mu -\lambda +\underset{%
x\notin A}{\bigwedge }(1-\Re (A)),$$ i.e., $$\underset{x\notin A}{\bigvee }\Re
(A)\leq 2-\mu -\lambda +\underset{S\overset{\sim }{\not\sqsubset }A}{\bigvee
}N_{x}^{P}(A)$$ for some $\delta >1-\lambda .$For any $t\in \lbrack 0,1],$ if
$\underset{x\notin A}{\bigvee }\Re (A)>t,$ then there exists $A_{\circ }$
such that $x_{\circ }\notin A_{\circ }$ and \ $\Re (A_{\circ })>t.$ \newline
Case 1. $t\leq 1-\lambda ,$ then $t\leq 2-\mu -\lambda +\underset{S\overset{%
\sim }{\not\sqsubset }A}{\bigvee }N_{x}^{P}(A).$\newline
Case 2. $t>1-\lambda .$ Here we set $\delta =\frac{1}{2}(t+1-\lambda )$ and
have $A_{\circ }\in \Re _{\delta },A_{\circ }\in \vartheta _{\delta }.$ In
addition, $$t<\Re (A_{\circ })\leq 1+F_{P}(A_{\circ })-\mu , \ t+\mu -1\leq
F_{P}(A_{\circ })=\tau _{P}(A_{\circ }^{c}).$$ Since $A_{\circ }\in \vartheta
_{\delta },$ we know that $S_{B}\in A_{\circ },$ $i.e.,$ $S_{B}\notin
A_{\circ }^{c}$ when $B\subseteq A_{\circ }$ and $S\overset{\sim }{%
\not\sqsubset }A_{\circ }^{c}.$ Therefore, $$2-\mu -\lambda +\underset{S%
\overset{\sim }{\not\sqsubset }A}{\bigvee }N_{x}^{P}(A)\geq 2-\mu
-\lambda +N_{x}^{P}(A_{\circ }^{c})\geq 2-\mu -\lambda +\tau
_{P}(A_{\circ }^{c})\geq t+(1-\lambda )\geq t.$$ By noticing that t is
arbitrary, we have completed the proof. \newline
(6) To prove that $[\beta _{5}]=[(X,\tau )\in \Gamma _{P}]$ see [1] Theorem
6.
\end{proof}

The above theorem is a generalization of the following corollary.

\begin{corollary}
The following are equivalent for a topological space $(X,\tau )$.\newline
(a) X is a strong compact space.\newline
(b) Every cover of \ X\ by members of a pre-subbase of $\tau _{P}$ has a
finite subcover.\newline
(c) Every universal net in X pre-converges to a point in X.\newline
(d) Each net in X has a subnet that pre-converges to some point in X.\newline
(e) Each net in X has a pre-adherent point.\newline
(f) Each family of pre-closed sets in X that has the finite intersection
property has a non-void intersection.
\end{corollary}

\begin{definition}
Let $\{(X_{s},\tau _{s}):s\in S\}$ be a family of fuzzifying topological
spaces, $\underset{s\in S}{\prod }X_{s}$ be the cartesian product of $%
\{X_{s}:s\in S\}$ and $\varphi =\{p_{s}^{-1}(U_{s}):s\in S,U_{s}\in
P(X_{s})\},$ where $p_{t}:\underset{s\in S}{\prod }X_{s}\rightarrow
X_{t}(t\in S)$ is a projection. For $\Phi \subseteq \varphi ,$ $S(\Phi )$
stands for the set of indices of elements in $\Phi $. The $P$-base $\beta
_{P}\in \Im (\underset{s\in S}{\prod }X_{s})$ of \ $\underset{s\in S}{\prod }%
(\tau _{P})_{s}$ is defined as

$$V\in \beta _{P}:=(\exists \Phi )(\Phi \Subset \varphi \wedge (\bigcap \Phi
=V))\rightarrow \forall s(s\in S(\Phi )\rightarrow V_{s}\in (\tau
_{P})_{s}), $$ i.e.,

$$\beta _{P}(V)=\underset{\Phi \Subset \varphi ,\bigcap \Phi =V}{\bigvee }\
\underset{s\in S(\Phi )}{\bigwedge }(\tau _{P})_{s}(V_{s}).$$
\end{definition}

\begin{definition}
Let $(X,\tau ),$ $(Y,\sigma )$\ be two fuzzifying topological space. A unary
fuzzy predicate $O_{P}\in F(Y^{X}),$ is called fuzzifying pre-openness, is
given as: $O_{P}(f):=\forall U(U\in \tau _{P}\rightarrow f(U)\in \sigma
_{P}).$ Intuitively, the degree to which f is pre-open is $$[O_{P}(f)]=%
\underset{B\subseteq X}{\bigwedge }\min (1,1-\tau _{P}(U)+\sigma _{P}\left(
f(U)\right) ).$$
\end{definition}

\begin{lemma}
Let $(X,\tau )$ and $(Y,\sigma )$\ be two fuzzifying topological space. For
any $f\in Y^{X},$\newline
$$O_{P}(f):=\forall B(B\in \beta _{P}^{X}\rightarrow f(B)\in \sigma _{P}),$$
where $\beta _{P}^{X}$ is a pre-base of $\tau _{P}.$
\end{lemma}

\begin{proof}
Clearly, $[O_{P}(f)]\leq \lbrack \forall U(U\in \beta _{P}^{X}\rightarrow
f(U)\in \sigma _{P})]$. Conversely, for any $U\subseteq X,$ we are going to
prove $$\min (1,1-\tau _{P}(U)+\sigma _{P}\left( f(U)\right) )\geq \lbrack
\forall V(V\in \beta _{P}^{X}\rightarrow f(V)\in \sigma _{P})].$$ If $\tau
_{P}(U)\leq \sigma _{P}(f(U)),$ it is hold clearly. Now assume $\tau
_{P}(U)>\sigma _{P}(f(U)).$ If $\Re \subseteq P(X)$ with $\bigcup \Re =U,$
then $\bigcup_{V\in \Re }f(V)=f(\bigcup \Re )=f(U).$ Therefore \newline
\begin{eqnarray*} \tau _{P}(U)-\sigma _{P}(f(U))&=& \underset{\Re \subseteq P(X),\bigcup \Re =U}{%
\bigvee }\ \underset{V\in \Re }{\bigwedge }\beta _{P}^{X}(V)-\underset{\wp
\subseteq P(Y),\bigcup \wp =f(U)}{\bigvee }\ \underset{W\in \wp }{\bigwedge }%
\sigma _{P}(W)\\ &\leq& \underset{\Re \subseteq
P(X),\bigcup \Re =U}{\bigvee }\ \underset{V\in \Re }{\bigwedge }\beta
_{P}^{X}(V)-\underset{\Re \subseteq P(X),\bigcup \Re =U}{\bigvee }\ \underset%
{V\in \Re }{\bigwedge }\sigma _{P}(f(V))\\ &\leq& \underset{\Re \subseteq
P(X),\bigcup \Re =U}{\bigvee }\ \underset{V\in \Re }{\bigwedge }\left( \beta
_{P}^{X}(V)-\sigma _{P}(f(V))\right) , \ \min (1,1-\tau _{P}(U)+\sigma _{P}\left( f(U)\right) )\\ &\geq& \underset{\Re
\subseteq P(X),\bigcup \Re =U}{\bigvee }\ \underset{V\in \Re }{\bigwedge }%
\min (1,1-\beta _{P}^{X}(V)+\sigma _{P}\left( f(V)\right) )\\ &\geq& \lbrack \forall V(V\in \beta _{P}^{X}\rightarrow f(V)\in \sigma _{P})].\end{eqnarray*}
\end{proof}

\begin{lemma}
For any family $\{(X_{s},\tau _{s}):s\in S\}$ of fuzzifying topological
spaces.\newline
(1) $\vDash (\forall s)(s\in S\rightarrow p_{s}\in O_{P});$ \\
(2) $\vDash
(\forall s)(s\in S\rightarrow p_{s}\in C_{P}).$
\end{lemma}

\begin{proof}
(1) For any $t\in S,$ we have $$O_{P}(p_{t})= \underset{U\in P(\underset{%
s\in S}{\prod }X_{s})}{\bigwedge }\min (1,1-\left( \underset{s\in S}{\prod }%
(\tau _{P})_{s}\right) (U)+(\tau _{P})_{t}\left( p_{t}(U)\right) ).$$ Then it
suffices to show that for any $U\in P(\underset{s\in S}{\prod }X_{s}),$ we
have $$(\tau _{P})_{t}\left( p_{t}(U)\right) \geq \left( \underset{s\in S}{%
\prod }(\tau _{P})_{s}\right) (U).$$ \newline
Assume $$\left( \underset{s\in S}{\prod }(\tau _{P})_{s}\right) (U)=\underset{%
\underset{\lambda \in \Lambda }{\cup }B_{\lambda }=U}{\bigvee }\ \underset{%
\lambda \in \Lambda }{\bigwedge }\ \underset{\Phi _{\lambda }\Subset \varphi
,\cap \Phi _{\lambda }=B_{\lambda }}{\bigvee }\ \underset{s\in S(\Phi
_{\lambda })}{\bigwedge }(\tau _{P})_{s}(V_{s})>\mu $$ where $$\Phi _{\lambda
}=\{p_{s}^{-1}(V_{s}):s\in S(\Phi _{\lambda })\}(\lambda \in \Lambda ).$$
\newline
Hence there exists $\{B_{\lambda }:\lambda \in \Lambda \}\subseteq P(%
\underset{s\in S}{\prod }X_{s})$ such that $\underset{\lambda \in \Lambda }{%
\bigcup }B_{\lambda }=U$ and furthermore, for any $\lambda \in \Lambda $,
there exists $\Phi _{\lambda }\Subset \varphi $ such that $\cap \Phi
_{\lambda }=B_{\lambda }$ and $\underset{s\in S(\Phi _{\lambda })}{\bigcap }%
p_{s}^{-1}(V_{s})=B_{\lambda },$ where for any $s\in S(\Phi _{\lambda })$ we
have $(\tau _{P})_{s}(V_{s})>\mu .$ Thus $$p_{t}(U)=p_{t}(\underset{\lambda
\in \Lambda }{\bigcup }\ \underset{s\in S(\Phi _{\lambda })}{\bigcap }%
p_{s}^{-1}(V_{s})).$$ \newline
(1) If for any $\lambda \in \Lambda $, $\underset{s\in S(\Phi _{\lambda })}{%
\bigcap }p_{s}^{-1}(V_{s})=\phi ,$ then $U=\phi ,p_{t}(U)=\phi $ and $(\tau
_{P})_{t}\left( p_{t}(U)\right) =1.$ Therefore $(\tau _{P})_{t}\left(
p_{t}(U)\right) \geq \left( \underset{s\in S}{\prod }(\tau _{P})_{s}\right)
(U).$ \newline
(2) If there exists $\lambda _{\circ }\in \Lambda ,$ such that $\phi \neq
\underset{s\in S(\Phi _{\lambda })}{\bigcap }p_{s}^{-1}(V_{s})=B_{\lambda
_{\circ }},$ \newline
(i) If $t\notin S(\Phi _{\lambda _{\circ }}),$ i.e., $t\in S-S(\Phi
_{\lambda _{\circ }})$, $p_{t}(B_{\lambda _{\circ }})=X_{t}.$ Therefore $%
(\tau _{P})_{t}(p_{t}(B_{\lambda _{\circ }}))=(\tau _{P})_{t}(X_{t})=1.$
\newline
(ii) If $t\in S(\Phi _{\lambda _{\circ }})$, then $p_{t}(B_{\lambda _{\circ
}})=V_{t}\subseteq X_{t}$. Thus \newline
$$p_{t}(U)=p_{t}((\underset{t\in S(\Phi _{\lambda _{\circ }})}{\bigcup }%
B_{\lambda _{\circ }})\cup (\underset{t\notin S(\Phi _{\lambda _{\circ }})}{%
\bigcup }B_{\lambda _{\circ }}))=(\underset{t\in S(\Phi _{\lambda _{\circ
}})}{\bigcup }p_{t}(B_{\lambda _{\circ }}))\cup (\underset{t\notin S(\Phi
_{\lambda _{\circ }})}{\bigcup }p_{t}(B_{\lambda _{\circ }}))=V_{t}\cup
X_{t}=X_{t}.$$ \newline
Hence $$(\tau _{P})_{t}(p_{t}(U))=(\tau _{P})_{t}(X_{t})=1 \ {\rm or} \ (\tau
_{P})_{t}(p_{t}(U))=(\tau _{P})_{t}(V_{t})>\lambda .$$ \newline
Therefore $(\tau _{P})_{t}\left( p_{t}(U)\right) \geq \left( \underset{s\in S%
}{\prod }(\tau _{P})_{s}\right) (U).$ Thus $O_{P}(p_{t})=1.$ \newline
(2) From Lemma 3.1 in [25] we have $\vDash (\forall s)(s\in S\rightarrow
p_{s}\in C).$ Furthermore, for any two fuzzifying topological spaces $%
(X,\tau )$ and $(Y,\sigma )$ and $f\in Y^{X},$ we have $C(f)\leq C_{P}(f)$
(Theorem 6.3 in [2]). Therefore $\vDash (\forall s)(s\in S\rightarrow
p_{s}\in C_{P}).$
\end{proof}

\begin{theorem}
Let $\{(X_{s},\tau _{s}):s\in S\}$ be the family of fuzzifying topological
spaces, then \newline
$$\vDash \exists U(U\subseteq \underset{s\in S}{\prod }X_{s}\wedge \Gamma
_{P}(U,\tau /U)\wedge \exists x(x\in Int_{P}(U))\rightarrow \exists
T(T\Subset S\wedge \forall t(t\in S-T\wedge \Gamma _{P}(X_{t},\tau _{t}))).$$
\end{theorem}

\begin{proof}
It suffices to show that \newline
$$\underset{U\in P(\underset{s\in S}{\prod }X_{s})}{\bigvee }\left( \Gamma
_{P}(U,\tau /U)\wedge \underset{x\in \underset{s\in S}{\prod }X_{s}}{\bigvee
}N_{x}^{P}(U)\right) \leq \underset{T\Subset S}{\bigvee }\ \underset{t\in S-T%
}{\bigwedge }\Gamma _{P}(X_{t},\tau _{t}).$$ Indeed, if \newline
$$\underset{U\in P(\underset{s\in S}{\prod }X_{s})}{\bigvee }\left( \Gamma
_{P}(U,\tau /U)\wedge \underset{x\in \underset{s\in S}{\prod }X_{s}}{\bigvee
}N_{x}^{P}(U)\right) >\mu >0,$$ then there exists $U\in P(\underset{s\in S}{%
\prod }X_{s})$ such that $\Gamma _{P}(U,\tau /U)>\mu $ and $\underset{x\in
\underset{s\in S}{\prod }X_{s}}{\bigvee }N_{x}^{P}(U)>\mu ,$ where $$%
N_{x}^{P}(U)=\underset{x\in V\subseteq U}{\bigvee }\ \left( \underset{s\in S}%
{\prod }(\tau _{P})_{s}\right) (V).$$ Furthermore, there exists $V$ such that
$x\in V\subseteq U$ and \ $\left( \underset{s\in S}{\prod }(\tau
_{P})_{s}\right) (V)>\mu .$ Since $\beta _{P}$ is a pre-base of $\underset{%
s\in S}{\prod }(\tau _{P})_{s},$ \newline
$$\left( \underset{s\in S}{\prod }(\tau _{P})_{s}\right) (V)=\underset{%
\underset{\lambda \in \Lambda }{\cup }B_{\lambda }=V}{\bigvee }\ \underset{%
\lambda \in \Lambda }{\bigwedge }\ \beta _{P}(B_{\lambda })= \underset{%
\underset{\lambda \in \Lambda }{\cup }B_{\lambda }=V}{\bigvee }\ \underset{%
\lambda \in \Lambda }{\bigwedge }\ \underset{\Phi _{\lambda }\Subset \varphi
,\cap \Phi _{\lambda }=B_{\lambda }}{\bigvee }\ \underset{s\in S(\Phi
_{\lambda })}{\bigwedge }(\tau _{P})_{s}(V_{s})>\mu ,$$ \newline
where $\Phi _{\lambda }=\{p_{s}^{-1}(V_{s}):s\in S(\Phi _{\lambda
})\}(\lambda \in \Lambda ).$\newline
Hence there exists $\{B_{\lambda }:\lambda \in \Lambda \}\subseteq P(%
\underset{s\in S}{\prod }X_{s})$ such that $\underset{\lambda \in \Lambda }{%
\cup }B_{\lambda }=V.$ Furthermore, for any $\lambda \in \Lambda $, there
exists $\Phi _{\lambda }\Subset \varphi $ such that $\cap \Phi _{\lambda
}=B_{\lambda }$ and for any $s\in S(\Phi _{\lambda }),$ we have $(\tau
_{P})_{s}(V_{s})>\mu .$ Since $x\in V,$ there exists $B_{\lambda _{x}}$ such
that $x\in B_{\lambda _{x}}\subseteq V\subseteq U.$ Hence there exists $\Phi
_{\lambda _{x}}\Subset \varphi $ such that $\cap \Phi _{\lambda
_{x}}=B_{\lambda _{x}}$ and $\underset{s\in S(\Phi _{\lambda })}{\bigcap }%
p_{s}^{-1}(V_{s})=B_{\lambda _{x}}\subseteq \underset{s\in S}{\prod }X_{s}$
and for any $s\in S(\Phi _{\lambda }),$ we have $(\tau
_{P})_{s}(V_{s})>1-\mu .$ By $\underset{s\in S(\Phi _{\lambda })}{\bigcap }%
p_{s}^{-1}(V_{s})=B_{\lambda _{x}},$ we have $P_{\delta }(B_{\lambda
_{x}})=V_{\delta }\subseteq X_{\delta },$if $\delta \in S(\Phi _{\lambda
_{x}});$ $P_{\delta }(B_{\lambda _{x}})=X_{\delta },$if $\delta \in S-S(\Phi
_{\lambda _{x}}).$ Since $B_{\lambda _{x}}\subseteq U,$ for any $\delta \in
S-S(\Phi _{\lambda _{x}}),$ we have $P_{\delta }(U)\supseteq P_{\delta
}(B_{\lambda _{x}})=X_{\delta }$ and $P_{\delta }(U)=X_{\delta }.$ On the
other hand, since for any $s\in S$ and $U_{s}\in P(X_{s}),\ \left( \underset{%
t\in S}{\prod }(\tau _{P})_{t}\right) \left( p_{s}^{-1}(U_{s})\right) \geq
(\tau _{P})_{s}(U_{s}),$ we have , for any $s\in S,$ $$I_{P}(p_{s})=\underset{%
U_{s}\in P(X_{s})}{\bigwedge }\min \left( 1,1-(\tau _{P})_{s}(U_{s})+%
\underset{t\in S}{\prod }(\tau _{P})_{t}\left( p_{s}^{-1}(U_{s})\right)
\right) =1.$$ Furthermore, since by Theorem 9 in [1], we have $$\vDash \Gamma
_{P}(X,\tau )\otimes I_{P}(f)\rightarrow \Gamma _{P}(f(X)),$$ then $$\Gamma
_{P}(U,\tau /U)=\Gamma _{P}(U,\tau /U)\otimes I_{P}(p_{s})\leq \Gamma
_{P}(P_{\delta }(U),\tau _{\delta })=\Gamma _{P}(X_{\delta },\tau _{\delta
}).$$ Therefore $$\underset{T\Subset S}{\bigvee }\ \underset{t\in S-T}{%
\bigwedge }\Gamma _{P}(X_{t},\tau _{t})\geq \underset{\delta \in S-S(\Phi
_{\lambda })}{\bigwedge }\Gamma _{P}(X_{\delta },\tau _{\delta })\geq \Gamma
_{P}(U,\tau /U)>\mu .$$
\end{proof}

The above theorem is a generalization of the following corollary.

\begin{corollary}
If there exists a coordinate pre-neighborhood strong compact subset U of
some point $x\in X$ of the product space, then all except a finite number of
coordinate spaces are strong compact.
\end{corollary}

\begin{lemma}
For any fuzzifying topological space $(X,\tau ),A\subseteq X,$ \newline
$\vDash T_{2}^{P}(X,\tau )\rightarrow T_{2}^{P}(A,\tau _{/A}).$
\end{lemma}

\begin{proof}
\begin{eqnarray*} [T_{2}^{P}(X,\tau )]&=& \underset{x,y\in X,x\neq y}{\bigwedge }\ \underset{%
U,V\in P(X),U\cap V=\phi }{\bigvee }(N_{x}^{P}(U),N_{y}^{P}(V))\\ &\leq& \underset{x,y\in X,x\neq y}{%
\bigwedge }\ \underset{(U\cap A)\cap (V\cap A)=\phi }{\bigvee }%
(N_{x}^{P^{A}}(U\cap A),N_{y}^{P^{A}}(V\cap A))\\ &\leq& \underset{x,y\in A,x\neq y}{%
\bigwedge }\ \underset{U^{\prime }\cap V^{\prime }=\phi ,U^{\prime
},V^{\prime }\in P(A)}{\bigvee }(N_{x}^{P^{A}}(U^{\prime
}),N_{y}^{P^{A}}(V^{\prime }))\\ &=& T_{2}^{P}(A,\tau _{/A}),\end{eqnarray*}
\newline
where $$N_{x}^{P^{A}}(U)=\underset{x\in C\subseteq U}{\bigvee }\tau _{P}/A(C) \ {\rm and } \ \tau _{P}/A(B)=\underset{B=V\cap A}{\bigvee }\tau _{P}(V).$$
\end{proof}

\begin{lemma}
For any fuzzifying $P-$topological space $(X,\tau ),$ \newline
$\vDash T_{2}^{P}(X,\tau )\otimes \Gamma _{P}(X,\tau )\rightarrow
T_{4}^{P}(X,\tau ).$ \newline
For the definition of $T_{4}^{P}(X,\tau )$ see [3, Definition 3.1 ].
\end{lemma}

\begin{proof}
If $[T_{2}^{P}(X,\tau )\otimes \Gamma _{P}(X,\tau )]=0$, then the result
holds. Now, suppose that $[T_{2}^{P}(X,\tau )\otimes \Gamma _{P}(X,\tau
)]>\lambda >0.$ Then $T_{2}^{P}(X,\tau )+\Gamma _{P}(X,\tau )-1>\lambda >0$.
Therefore from Theorem 10 in [1]\newline
$T_{2}^{P}(X,\tau )\otimes (\Gamma _{P}(A)\wedge \Gamma _{P}(B))\wedge
(A\cap B=\phi )\vDash ^{ws}T_{2}^{P}(X,\tau )\rightarrow $ ($\exists
U)(\exists V)((U\in \tau _{P})\wedge (V\in \tau _{P})\wedge (A\subseteq
U)\wedge (B\subseteq V)\wedge (A\cap B=\phi )).$ Then for any $A,B\subseteq
X,A\cap B=\phi ,$ \newline
$$T_{2}^{P}(X,\tau )\otimes (\Gamma _{P}(A)\wedge \Gamma _{P}(B))\leq
\underset{U\cap V=\phi ,A\subseteq U,B\subseteq V}{\bigvee }\min (\tau
_{P}(U),\tau _{P}(V))$$
or equivalently $$T_{2}^{P}(X,\tau )\leq \Gamma _{P}(A)\wedge \Gamma
_{P}(B)\rightarrow \underset{U\cap V=\phi ,A\subseteq U,B\subseteq V}{%
\bigvee }\min (\tau _{P}(U),\tau _{P}(V))$$ \newline
Hence for any $A,B\subseteq X,A\cap B=\phi $, $$1-[\Gamma _{P}(A)\wedge
\Gamma _{P}(B)]+\underset{U\cap V=\phi ,A\subseteq U,B\subseteq V}{\bigvee }%
\min (\tau _{P}(U),\tau _{P}(V))+\Gamma _{P}(X,\tau )-1>\lambda .$$ From
Theorem 7 in [1] we have $$\vDash \Gamma _{P}(X,\tau )\otimes A\in \digamma
_{P}\rightarrow \Gamma _{P}(A).$$ Then \begin{eqnarray*} \Gamma _{P}(X,\tau )+[\tau
_{P}(A^{c})\wedge \tau _{P}(B^{c})]-1&=& (\Gamma _{P}(X,\tau )+\tau
_{P}(A^{c})-1)\wedge (\Gamma _{P}(X,\tau )+\tau _{P}(B^{c})-1)\\ &\leq&
(\Gamma _{P}(X,\tau )\otimes \tau _{P}(A^{c}))\wedge (\Gamma _{P}(X,\tau
)\otimes \tau _{P}(B^{c}))\\ &\leq& [\Gamma _{P}(A)\wedge \Gamma _{P}(B)].\end{eqnarray*}
Thus $$\Gamma _{P}(X,\tau )-[\Gamma _{P}(A)\wedge \Gamma _{P}(B)]-1\leq
-[\tau _{P}(A^{c})\wedge \tau _{P}(B^{c})].$$ So, $$1-[\tau _{P}(A^{c})\wedge
\tau _{P}(B^{c})]+\underset{U\cap V=\phi ,A\subseteq U,B\subseteq V}{\bigvee
}\min (\tau _{P}(U),\tau _{P}(V))>\lambda ,$$ i.e., $$T_{4}^{P}(X,\tau )=%
\underset{A\cap B=\phi }{\bigwedge }\min (1,1-[\tau _{P}(A^{c})\wedge \tau
_{P}(B^{c})]+ \underset{U\cap V=\phi ,A\subseteq U,B\subseteq V}{\bigvee }%
\min (\tau _{P}(U),\tau _{P}(V)))>\lambda .$$
\end{proof}

The above lemma is a generalization of the following corollary.

\begin{corollary}
Every strong compact pre-Hausdorff topological space is pre-normal.
\end{corollary}

\begin{lemma}
For any fuzzifying $P-$topological space $(X,\tau ),$ \newline
$\vDash T_{2}^{P}(X,\tau )\otimes \Gamma _{P}(X,\tau )\rightarrow
T_{3}^{P}(X,\tau ).$ For the definition of \newline
$T_{3}^{P}(X,\tau )$ see [3, Definition 3.1 ].
\end{lemma}

\begin{proof}
Immediate, set $A=\{x\}$\ in the above lemma.
\end{proof}

The above lemma is a generalization of the following corollary.

\begin{corollary}
Every strong compact pre-Hausdorff topological space is pre-regular.
\end{corollary}

\begin{theorem}
For any fuzzifying topological space $(X,\tau )$ and $A\subseteq X,$\newline
$\vDash T_{2}^{P}(X,\tau )\otimes \Gamma _{P}(A)\rightarrow A\in \digamma
_{P}.$
\end{theorem}

\begin{proof}
For any$\{x\}\subset A^{c},$ we have $\{x\}\cap A=\phi $ and $\Gamma
_{P}(\{x\})=1.$ By Theorem 10 in [1] \newline
$$[T_{2}^{P}(X,\tau )\otimes (\Gamma _{P}(A)\wedge \Gamma _{P}(\{x\}))]\leq
\underset{G\cap H_{x}=\phi ,A\subseteq G,x\in H_{x}}{\bigvee }\min (\tau
_{P}(G),\tau _{P}(H_{x}))).$$ Assume $$\beta _{x}=\{H_{x}:A\cap H_{x}=\phi
,x\in H_{x}\},\ \underset{x\in X\backslash A}{\bigcup }f(x)\supseteq A^{c}$$
and $$\underset{x\in A^{c}}{\bigcup }f(x)\cap A=\underset{x\in A^{c}}{\bigcup
}(f(x)\cap A)=\phi .$$ So, $\underset{x\in A^{c}}{\bigcup }f(x)=A^{c}.$%
\newline
Therefore \begin{eqnarray*} [T_{2}^{P}(X,\tau )\otimes \Gamma _{P}(A)]&\leq& \underset{G\cap
H_{x}=\phi ,A\subseteq G,x\in H_{x}}{\bigvee }\tau _{P}(H_{x}) \\ &\leq&
\underset{x\in A^{c}}{\bigwedge }\ \underset{A\cap H_{x}=\phi ,x\in H_{x}}{%
\bigvee }\tau _{P}(H_{x})\\ &=& \underset{f\in \underset{x\in A^{c}}{\prod }\beta _{x}}{\bigvee }\
\underset{x\in A^{c}}{\bigwedge }\tau _{P}(f(x)) \\ &\leq& \underset{f\in
\underset{x\in A^{c}}{\prod }\beta _{x}}{\bigvee }\ \tau _{P}(\underset{x\in
A^{c}}{\bigcup }f(x)) \\ &=& \underset{f\in \underset{x\in X\backslash A}{\prod }%
\beta _{x}}{\bigvee }\ \tau _{P}(A^{c}) = \digamma _{P}(A).\end{eqnarray*}
\end{proof}

The above theorem is a generalization of the following corollary.

\begin{corollary}
Strong compact subspace of a pre-Hausdorff topological space is pre-closed.
\end{corollary}

\section{Fuzzifying \textbf{locally strong compactness }}

\begin{definition}
Let $\Omega $ be a class of fuzzifying topological spaces. A unary fuzzy
predicate $L_{P}C\in \Im (\Omega ),$ called fuzzifying locally strong
compactness, is given as follows:\newline
$$(X,\tau )\in L_{P}C:=(\forall x)(\exists B)((x\in Int_{P}(B)\otimes \Gamma
_{P}(B,\tau /B)).$$ Since $[x\in Int_{P}(X)]=N_{x}^{P}(X)=1,$ then $%
L_{P}C(X,\tau )\geq \Gamma _{P}(X,\tau ).$ Therefore, $\vDash (X,\tau )\in
\Gamma _{P}\rightarrow (X,\tau )\in L_{P}C.$\newline
Also, since $\vDash (X,\tau )\in \Gamma \rightarrow (X,\tau )\in LC$ [21]and
$\vDash (X,\tau )\in \Gamma _{P}\rightarrow (X,\tau )\in \Gamma $ [1], $%
\vDash (X,\tau )\in \Gamma _{P}\rightarrow (X,\tau )\in LC$.
\end{definition}

\begin{theorem}
For any fuzzifying topological space $(X,\tau )$ and $A\subseteq X,$\newline
$\vDash (X,\tau )\in L_{P}C\otimes A\in F_{P}\rightarrow (A,\tau /A)\in
L_{P}C.$
\end{theorem}

\begin{proof}
We have $$L_{P}C(X,\tau )=\underset{x\in X}{\bigwedge }\ \underset{B\subseteq
X}{\bigvee }\max (0,N_{x}^{P^{X}}(B)+\Gamma _{P}(B,\tau /B)-1)$$ and $$%
L_{P}C(A,\tau /A)=\underset{x\in A}{\bigwedge }\ \underset{G\subseteq A}{%
\bigvee }\max (0,N_{x}^{P^{A}}(G)+\Gamma _{P}(G,(\tau /A)/G)-1).$$ Now,
suppose that $[(X,\tau )\in L_{P}C\otimes A\in F_{P}]>\lambda >0.$ Then for
any $x\in A,$ there exists $B\subseteq X$ such that \newline
\begin{eqnarray} N_{x}^{P^{X}}(B)+\Gamma _{P}(B,\tau /B)+\tau _{P}(X-A)-2>\lambda \label{e1}\end{eqnarray}
\newline
Set $E=A\cap B\in P(A).$ Then $$N_{x}^{P^{A}}(E)=\underset{E=C\cap B}{\bigvee
}N_{x}^{P^{X}}(C)\geq N_{x}^{P^{X}}(B)$$ and for any $U\in P(E)$, we have \begin{eqnarray*}
(\tau _{P}/A)_{P}/E(U)&=&\underset{U=C\cap E}{\bigvee }\tau _{P}/A(C)\\ &=&
\underset{U=C\cap E}{\bigvee }\ \underset{C=D\cap A}{\bigvee }\tau _{P}(D)\\ &=&
\underset{U=D\cap A\cap E}{\bigvee }\tau _{P}(D)=\underset{U=D\cap E}{%
\bigvee }\tau _{P}(D).\end{eqnarray*} Similarly, $$(\tau _{P}/B)_{P}/E(U)=\underset{U=D\cap
E}{\bigvee }\tau _{P}(D).$$ Thus, $(\tau _{P}/B)_{P}/E=(\tau _{P}/A)_{P}/E$
and $\Gamma _{P}(E,(\tau /A)/E)=\Gamma _{P}(E,(\tau /B)/E).$ Furthermore, \begin{eqnarray*}
[E\in F_{P}/B]&=&\tau _{P}/B(B-E)=\tau _{P}/B(B\cap E^{c})\\ &=& \underset{%
B\cap E^{c}=B\cap D}{\bigvee }\tau _{P}(D)\\ &\geq& \tau _{P}(X-A)=F_{P}(A).\end{eqnarray*}
Since $\vDash (X,\tau )\in \Gamma _{P}\otimes A\in F_{P}\rightarrow (A,\tau
/A)\in \Gamma _{P}$ (see [1], Theorem 7], from (\ref{e1}) we have for any $x\in A$
that
\begin{eqnarray*}\underset{G\subseteq A}{\bigvee }\max (0,N_{x}^{P^{A}}(G)+\Gamma
_{P}(G,(\tau /A)/G)-1)&\geq& N_{x}^{P^{A}}(E)+\Gamma _{P}(E,(\tau /A)/E)-1 \\ &=&
N_{x}^{P^{A}}(E)+\Gamma _{P}(E,(\tau /B)/E)-1\\ &\geq&
N_{x}^{P^{X}}(B)+[\Gamma _{P}(B,\tau /B)\otimes E\in F_{P}/B]-1 \\ &\geq& N_{x}^{P^{X}}(B)+\Gamma _{P}(B,\tau /B)+[E\in F_{P}/B]-2\\ &\geq& N_{x}^{P^{X}}(B)+\Gamma _{P}(B,\tau /B)+[A\in
F_{P}]-2 >\lambda .\end{eqnarray*}
Therefore $$L_{P}C(A,\tau /A)=\underset{x\in A}{\bigwedge }\ \underset{%
G\subseteq A}{\bigvee }\max (0,N_{x}^{P^{A}}(G)+\Gamma _{P}(G,(\tau
/A)/G)-1)>\lambda .$$
Hence $[(X,\tau )\in L_{P}C\otimes A\in F_{P}]\leq L_{P}C(A,\tau /A).$
\end{proof}

As a crisp result of the above theorem we have the following corollary.

\begin{corollary}
Let $A$ be a pre-closed subset of locally strong compact space $(X,\tau ).$
Then $A$ with the relative topology $\tau /A$ is locally strong compact.
\end{corollary}

The following theorem is a generalization of the statement "If $X$ is a
pre-Hausdorff topological space and $A$ is a pre-dense locally strong
compact subspace, then $A$ is pre-open", where $A$ is a pre-dense in a
topological space $X$ if and only if the pre-closure of $A$ is $X$.

\begin{theorem}
For any fuzzifying $P$-topological space $(X,\tau )$ and $A\subseteq X,$%
$$\vDash T_{2}^{P}(X,\tau )\otimes L_{P}C(A,\tau /A)\otimes (Cl_{P}(A)\equiv
X)\rightarrow A\in \tau _{P}.$$
\end{theorem}

\begin{proof}
Assume $$[T_{2}^{P}(X,\tau )\otimes L_{P}C(A,\tau /A)\otimes (Cl_{P}(A)\equiv
X)]>\lambda >0.$$ Then
$$L_{P}C(A,\tau /A)>\lambda -[T_{2}^{P}(X,\tau )\otimes (Cl_{P}(A)\equiv
X)]+1=\lambda ^{\prime }>\lambda $$ ,i. e., \newline
$$\underset{x\in A}{\bigwedge }\ \underset{B\subseteq A}{\bigvee }\max
(0,N_{x}^{P^{A}}(B)+\Gamma _{P}(B,(\tau /A)/B)-1)>\lambda ^{\prime }.$$ Thus
for any $x\in A$, there exists $B_{x}\subseteq A$ such that $$%
N_{x}^{P^{A}}(B_{x})+\Gamma _{P}(B_{x},(\tau /A)/B_{x})-1>\lambda ^{\prime
}. $$ i.e., $$\underset{H\cap A=B_{x}}{\bigvee }\ \underset{x\in K\subseteq H}{%
\bigvee }\tau _{P}(K)+\Gamma _{P}(B_{x},(\tau /A)/B_{x})-1>\lambda
^{\prime }.$$ Hence there exists $K_{x}$ such that $$K_{x}\cap A=B_{x},\ \tau
_{P}(K_{x})+\Gamma _{P}(B_{x},(\tau /A)/B_{x})-1>\lambda ^{\prime }.$$
Therefore $\tau _{P}(K_{x})>\lambda ^{\prime }.$\newline
(1) If for any $x\in A$ there exists $K_{x}$ such that $$x\in K_{x}\subseteq
B_{x}\subseteq A,\ {\rm then} \ \underset{x\in A}{\bigcup }K_{x}=A$$ and $$\tau
_{P}(A)=\tau _{P}(\underset{x\in A}{\bigcup }K_{x})\geq \underset{x\in A}{%
\bigwedge }\tau _{P}(K_{x})\geq \lambda ^{\prime }>\lambda .$$
(2) If there exists $x_{\circ }\in A$ such that $$K_{x_{\circ }}\cap
(B_{x_{\circ }}^{c})\neq \phi ,\ \tau _{P}(K_{x_{\circ }}) +\Gamma
_{P}(B_{x_{\circ }},(\tau /A)/B_{x_{\circ }})-1>\lambda ^{\prime }.$$ From
the hypothesis $$[T_{2}^{P}(X,\tau )\otimes L_{P}C(A,\tau /A)\otimes
(Cl_{P}(A)\equiv X)]>\lambda >0,$$ we have $[T_{2}^{P}(X,\tau )\otimes
(Cl_{P}(A)\equiv X)]\neq 0$. So $$\tau _{P}(K_{x_{\circ }})+\Gamma
_{P}(B_{x_{\circ }},(\tau /A)/B_{x_{\circ }})-1+[T_{2}^{P}(X,\tau
)\otimes (Cl_{P}(A)\equiv X)]-1>0.$$ Therefore $$\tau _{P}(K_{x_{\circ }})
+\Gamma _{P}(B_{x_{\circ }},(\tau /A)/B_{x_{\circ }})-1+T_{2}^{P}(X,\tau
)+[(Cl_{P}(A)\equiv X)]-1-1>\lambda .$$ Since \begin{eqnarray*} (\tau _{P}/A)_{P}/B_{x_{\circ
}}(U)&=& \underset{U=C\cap B_{x_{\circ }}}{\bigvee }\tau _{P}/A(C)\\ &=&\underset{%
U=C\cap B_{x_{\circ }}}{\bigvee }\underset{C=D\cap A}{\bigvee }\tau _{P}(D)\\  &=&\underset{U=D\cap B_{x_{\circ }}}{\bigvee }\tau _{P}(D) =\tau
_{P}/B_{x_{\circ }}(U), \end{eqnarray*} $$\Gamma _{P}(B_{x_{\circ }},(\tau /A)/B_{x_{\circ
}})=\Gamma _{P}(B_{x_{\circ }},\tau /B_{x_{\circ }}).$$ From Theorem 3.3 we
have \begin{eqnarray*}\tau _{P}(B_{x_{\circ }}^{c})&\geq& T_{2}^{P}(X,\tau )\otimes \Gamma
_{P}(B_{x_{\circ }},\tau /B_{x_{\circ }})\\ &\geq& T_{2}^{P}(X,\tau )+\Gamma
_{P}(B_{x_{\circ }},\tau /B_{x_{\circ }})-1.\end{eqnarray*} Hence $$\tau _{P}(K_{x_{\circ
}})+\tau _{P}(B_{x_{\circ }}^{c})+[Cl_{P}(A)\equiv X]-2>\lambda .$$ Now, for
any $y\in A^{c}$ we have $$[Cl_{P}(A)\equiv X]=\underset{x\in X}{\bigwedge }%
(1-N_{x}^{P^{X}}(A^{c}))\leq 1-N_{y}^{P^{X}}(A^{c}).$$ Since $(X,\tau )$ is a
fuzzifying $P$-topological space, \begin{eqnarray*} \tau _{P}(K_{x_{\circ }})+\tau
_{P}(B_{x_{\circ }}^{c})-1 &\leq& \tau _{P}(K_{x_{\circ }})\otimes \tau
_{P}(B_{x_{\circ }}^{c}) \\ &\leq& \tau _{P}(K_{x_{\circ }})\wedge \tau
_{P}(B_{x_{\circ }}^{c}) \\ &\leq& \tau _{P}(K_{x_{\circ }}\cap B_{x_{\circ
}}^{c}) \\ &\leq& N_{y}^{P^{X}}(K_{x_{\circ }}\cap B_{x_{\circ }}^{c})\leq
N_{y}^{P^{X}}(A^{c}),\end{eqnarray*} where \begin{eqnarray*}y\in K_{x_{\circ }}\cap B_{x_{\circ
}}^{c}\subseteq H_{x_{\circ }}\cap (H_{x_{\circ }}\cap A)^{c}&=& H_{x_{\circ
}}\cap (H_{x_{\circ }}^{c}\cup A^{c})\\ &=& H_{x_{\circ }}\cap A^{c}\subseteq
A^{c}.\end{eqnarray*} Therefore \begin{eqnarray*} 0<\lambda <\tau _{P}(K_{x_{\circ }})+\tau
_{P}(B_{x_{\circ }}^{c})+[Cl_{P}(A)\equiv X]-2&=& \tau _{P}(K_{x_{\circ
}})+\tau _{P}(B_{x_{\circ }}^{c})-1+[Cl_{P}(A)\equiv X]-1 \\ &\leq&
N_{y}^{P^{X}}(A^{c})+1-N_{y}^{P^{X}}(A^{c})-1=0,\end{eqnarray*} a contradiction. So, case
(2) does not hold. We complete the proof.
\end{proof}
\begin{theorem}
For any fuzzifying $P$-topological space $(X,\tau ),$\newline
$$\vDash T_{2}^{P}(X,\tau )\otimes (L_{P}C(X,\tau ))^{2}\rightarrow \forall
x\forall U(U\in N_{x}^{P^{X}}\rightarrow \exists V(V\in N_{x}^{P^{X}}\wedge Cl_{P}(V)\subseteq U\wedge \Gamma _{P}(V))),$$
where $(L_{P}C(X,\tau ))^{2}:=L_{P}C(X,\tau )\otimes L_{P}C(X,\tau ).$
\end{theorem}

\begin{proof}
We need to show that for any $x$ and $U$, $x\in U$, \newline
$$T_{2}^{P}(X,\tau )\otimes (L_{P}C(X,\tau ))^{2}\otimes N_{x}^{P^{X}}(U)\leq
\underset{V\subseteq X}{\bigvee }(N_{x}^{P^{X}}(V)\wedge \underset{y\in
U^{c}}{\bigwedge }N_{x}^{P^{X}}(V^{c})\wedge \Gamma _{P}(V,\tau /V)).$$
Assume that $T_{2}^{P}(X,\tau )\otimes (L_{P}C(X,\tau ))^{2}\otimes
N_{x}^{P^{X}}(U)>\lambda >0.$ Then for any $x\in X$ there exists $C$ such
that \newline
\begin{eqnarray}T_{2}^{P}(X,\tau )+N_{x}^{P^{X}}(C)+(L_{P}C(X,\tau
))^{2}+N_{x}^{P^{X}}(U)-3>\lambda .\label{eq1} \end{eqnarray}
Since $(X,\tau )$ is fuzzifying $P$-topological space, \begin{eqnarray*}
N_{x}^{P^{X}}(C)+N_{x}^{P^{X}}(U)-1\leq N_{x}^{P^{X}}(C)\otimes
N_{x}^{P^{X}}(U) &\leq& N_{x}^{P^{X}}(C)\wedge N_{x}^{P^{X}}(U)\\ &\leq&
N_{x}^{P^{X}}(C\cap U)=\underset{x\in W\subseteq C\cap U}{\bigvee }\tau
_{P}(W).\end{eqnarray*} Therefore there exists $W$ such that $x\in W\subseteq C\cap U,$
and $T_{2}^{P}(X,\tau )+(L_{P}C(X,\tau ))^{2}+\tau _{P}(W)-2>\lambda .$ By
Lemmas 3.3 and 3.5 we have $T_{2}^{P}(X,\tau )\leq T_{2}^{P}(C,\tau /C)$ and
$$T_{2}^{P}(C,\tau /C)+\Gamma _{P}(C,\tau /C)-1\leq T_{2}^{P}(C,\tau
/C)\otimes \Gamma _{P}(C,\tau /C)\leq T_{3}^{P}(C,\tau /C).$$ Thus $%
T_{3}^{P}(X,\tau )+\Gamma _{P}(C,\tau /C)+\tau _{P}(W)-2>\lambda .$ Since
for any $x\in W\subseteq U,$ we have $$T_{3}^{P}(C,\tau /C)\leq 1-\tau
_{P}/C(W)+\underset{G\subseteq C}{\bigvee }\left( (N_{x}^{P^{C}}(G)\wedge
\underset{y\in C-W}{\bigwedge }N_{y}^{P^{C}}(C-G))\right) ,$$ so there exists
$G,x\in G\subseteq W$ such that \begin{eqnarray*}\ \left( (N_{x}^{P^{C}}(G)\wedge \underset{%
y\in C-W}{\bigwedge }N_{y}^{P^{C}}(C-G))\right) &\geq& T_{3}^{P}(C,\tau
/C)+\tau _{P}/C(W)-1\geq T_{3}^{P}(C,\tau /C)+\tau _{P}(W)-1\end{eqnarray*} and $$\left(
(N_{x}^{P^{C}}(G)\wedge \underset{y\in C-W}{\bigwedge }N_{y}^{P^{C}}(C-G))%
\right) +\Gamma _{P}(C,\tau /C)-1>\lambda .$$ Thus $$N_{x}^{P^{C}}(G)=\underset%
{D\cap C=G}{\bigvee }N_{x}^{P^{X}}(D)=N_{x}^{P^{X}}(G\cup C^{c})>\lambda
^{\prime }=\lambda +1-\Gamma _{P}(C,\tau /C)\geq \lambda .$$ Furthermore, for
any $y\in C-W,$ $$N_{y}^{P^{C}}(C-G)=\underset{D\cap C=C\cap G^{c}}{\bigvee }%
N_{y}^{P^{X}}(G^{c}\cup C^{c})=N_{y}^{P^{X}}(G^{c})>\lambda ^{\prime }$$ and $$%
N_{x}^{P^{X}}(G)=N_{x}^{P^{X}}((G\cup C^{c})\cap C)\geq N_{x}^{P^{X}}(G\cup
C^{c})\wedge N_{x}^{P^{X}}(C)>\lambda ^{\prime }.$$ Since $%
N_{y}^{P^{X}}(G^{c})=\underset{x\in B^{c}\subseteq G^{c}}{\bigvee }\tau
_{P}(B^{c})>\lambda ^{\prime },$ for any $y\in C-W,$ there exists $B_{y}^{c}$
such that $y\in B_{y}^{c}\subseteq G^{c}$ and $\tau _{P}(B_{y}^{c})>\lambda
^{\prime }.$ Set $B^{c}=\underset{y\in C-W}{\bigcup }B_{y}^{c}\ .$ Then $%
C-W\subseteq B^{c}\subseteq G^{c}$ and $$\tau _{P}(B^{c})\geq \underset{y\in
C-W}{\bigwedge }\tau _{P}(B_{y}^{c})\geq \lambda ^{\prime }.$$ Again, set $%
V=B\cap C,$ then $V\subseteq (C-W)^{c}\cap C=(C^{c}\cup W)\cap C=C\cap
W=W\subseteq U\cap C$ and $V^{c}=B^{c}\cup C^{c}.$ Since $(X,\tau )$ is
fuzzifying $P$-topological space, \newline
\begin{eqnarray}N_{x}^{P^{X}}(V)=N_{x}^{P^{X}}(B\cap C)&\geq& N_{x}^{P^{X}}(B)\wedge
N_{x}^{P^{X}}(C) \nonumber \\ &\geq& N_{x}^{P^{X}}(G)\wedge N_{x}^{P^{X}}(C)>\lambda . \label{eq3} \end{eqnarray}

By (\ref{eq3}) and Theorem 3.3, $$\tau _{P}(C^{c})\geq T_{2}^{P}(X,\tau )\otimes
\Gamma _{P}(C,\tau /C) \geq T_{2}^{P}(X,\tau )+\Gamma _{P}(C,\tau
/C)-1\geq \lambda ^{\prime }.$$ So $$\tau _{P}(V^{c})=\tau _{P}(B^{c}\cup
C^{c})\geq \tau _{P}(B^{c})\wedge \tau _{P}(C^{c})\geq \lambda ^{\prime
}, $$ i.e., $\tau _{P}(V^{c})+\Gamma _{P}(C,\tau /C)-1\geq \lambda $ and \begin{eqnarray}
\Gamma _{P}(V,\tau /V)=\Gamma _{P}(V,(\tau /C)/V)\nonumber \\ &\geq& \tau
_{P}/C(C-V)+\Gamma _{P}(C,\tau /C)-1 \nonumber \\ &\geq& \tau _{P}(V^{c})+\Gamma _{P}(C,\tau /C)-1
\geq \lambda \label{eq4} \end{eqnarray}
Finally, \begin{eqnarray} \underset{y\in U^{c}}{\bigwedge }N_{y}^{P^{X}}(V^{c})\geq \underset%
{y\in V^{c}}{\bigwedge }N_{y}^{P^{X}}(V^{c})=\tau _{P}(V^{c})\geq \lambda  \label{eq5} \end{eqnarray}
Thus by (\ref{eq3}), (\ref{eq4}) and (\ref{eq5}), for any $x\in U$, there exists $V\subseteq U$ such
that $N_{x}^{P^{X}}(V)>\lambda ,$ $\underset{y\in U^{c}}{\bigwedge }%
N_{y}^{P^{X}}(V^{c})\geq \lambda $ and $\Gamma _{P}(V,\tau /V)\geq \lambda .$
So $\underset{V\subseteq X}{\bigvee }(N_{x}^{P^{X}}(V)\wedge \underset{y\in
U^{c}}{\bigwedge }N_{y}^{P^{X}}(V^{c})\wedge \Gamma _{P}(V,\tau /V))\geq
\lambda .$
\end{proof}
\begin{theorem}
For any fuzzifying $P$-topological space $(X,\tau ),$
$$\vDash T_{2}^{P}(X,\tau )\otimes (L_{P}C(X,\tau ))^{2}\rightarrow
T_{3}^{P}(X,\tau ).$$
\end{theorem}

\begin{proof}
By Theorem 4.3, for any $x\in U,$ we have \newline
$$\underset{x\in V\subseteq U}{\bigvee }(N_{x}^{P^{X}}(V)\wedge \underset{%
y\in U^{c}}{\bigwedge }N_{y}^{P^{X}}(V^{c})\geq \lbrack T_{2}^{P}(X,\tau
)\otimes (\Gamma _{P}(C,\tau /C))^{2}\otimes N_{x}^{P^{X}}(U)].$$
Thus $$1-N_{x}^{P^{X}}(U)+\underset{x\in V\subseteq U}{\bigvee }%
(N_{x}^{P^{X}}(V)\wedge \underset{y\in U^{c}}{\bigwedge }%
N_{y}^{P^{X}}(V^{c})\geq \lbrack T_{2}^{P}(X,\tau )\otimes (\Gamma
_{P}(C,\tau /C))^{2}],$$
i.e., $$[T_{3}^{P}(X,\tau )]\geq \lbrack T_{2}^{P}(X,\tau )\otimes (\Gamma
_{P}(C,\tau /C))^{2}].$$
\end{proof}
\begin{theorem}
For any fuzzifying $P$-topological space $(X,\tau ),$
\begin{eqnarray*} \vDash T_{3}^{P}(X,\tau )\otimes L_{P}C(X,\tau )&\rightarrow& \forall
A\forall U(U\in N_{A}^{P^{X}}\otimes \Gamma _{P}(A,\tau /A)\\ &\rightarrow&  \exists
V(V\subseteq U\wedge U\in N_{A}^{P^{X}}\wedge \tau _{P}(V^{c})\wedge \Gamma
_{P}(V,\tau /V))),\end{eqnarray*}
where $$U\in N_{A}^{P^{X}}:=(\forall x)(x\in A\wedge U\in N_{x}^{P^{X}}).$$
\end{theorem}
\begin{proof}
We only need to show that for any $A,U\in P(X),$
$$[T_{3}^{P}(X,\tau )\otimes L_{P}C(X,\tau )\otimes \Gamma _{P}(A,\tau
/A)\otimes N_{A}^{P^{X}}(U)]\leq \underset{V\subseteq U}{\bigvee }
(N_{A}^{P^{X}}(V)\wedge \tau _{P}(V^{c})\wedge \Gamma _{P}(V,\tau /V)).$$

Indeed, if $$[T_{3}^{P}(X,\tau )\otimes L_{P}C(X,\tau )\otimes \Gamma
_{P}(A,\tau /A)\otimes N_{A}^{P^{X}}(U)]>\lambda >0,$$ then for any $x\in A$,
there exists $C\in P(X)$ such that $$[T_{3}^{P}(X,\tau )\otimes
N_{x}^{P^{X}}(C)\otimes \Gamma _{P}(C,\tau /C)\otimes \Gamma _{P}(A,\tau
/A)\otimes N_{A}^{P^{X}}(U)]>\lambda .$$

Since $(X,\tau )$ is fuzzifying $P$-topological space,\begin{eqnarray*} \underset{x\in
W\subseteq C\cap U}{\bigvee }\tau _{P}(W)=N_{x}^{P^{X}}(C\cap U) &\geq&
N_{x}^{P^{X}}(C)\wedge N_{x}^{P^{X}}(U) \\ && \geq N_{x}^{P^{X}}(C)\wedge
N_{A}^{P^{X}}(U)\\ && \geq N_{x}^{P^{X}}(C)\otimes N_{A}^{P^{X}}(U).\end{eqnarray*}

Then there
exists $W$ such that $x\in W\subseteq C\cap U,$ and $$[T_{3}^{P}(X,\tau
)\otimes \tau _{P}(W)\otimes \Gamma _{P}(C,\tau /C)\otimes \Gamma
_{P}(A,\tau /A)]>\lambda .$$ Therefore
\begin{eqnarray}[T_{3}^{P}(X,\tau )]+\tau _{P}(W)-1>\lambda +2-\Gamma _{P}(C,\tau
/C)-\Gamma _{P}(A,\tau /A)]=\lambda ^{\prime }\geq \lambda .\end{eqnarray}
Since for any $x\in W,$
$$[T_{3}^{P}(X,\tau )]\leq 1-\tau _{P}(W)+\underset{%
B\subseteq W}{\bigvee }(N_{x}^{P^{X}}(B)\wedge \underset{y\in W^{c}}{%
\bigwedge }N_{y}^{P^{X}}(B^{c})),$$ we have $$\underset{B\subseteq W}{\bigvee }%
(N_{x}^{P^{X}}(B)\wedge \underset{y\in W^{c}}{\bigwedge }%
N_{y}^{P^{X}}(B^{c}))>\lambda ^{\prime }.$$ Thus there exists $B_{x}$ such
that $x\in B_{x}\subseteq W\subseteq C\cap U$ and for any $y\in W^{c},$ we
have $$N_{y}^{P^{X}}(B_{x}^{c})>\lambda ^{\prime }, \ \ N_{x}^{P^{X}}(B_{x})>%
\lambda ^{\prime }.$$ Since $$N_{y}^{P^{X}}(B_{x}^{c})=\underset{x\in
G^{c}\subseteq B_{x}^{c}}{\bigvee }\tau _{P}(G^{c})>\lambda ^{\prime },$$
then for any $y\in W^{c}$, there exists $G_{xy}$ such that $$x\in
G_{xy}^{c}\subseteq B_{x}^{c} \ {\rm and } \ \tau _{P}(G_{xy}^{c})>\lambda ^{\prime
}. $$ Set $$G_{x}^{c}=\underset{y\in W^{c}}{\bigcup }G_{xy}^{c},$$ then $%
W^{c}\subseteq G_{xy}^{c}\subseteq B_{x}^{c}$ and $\tau _{P}(G_{x}^{c})\geq
\underset{y\in W^{c}}{\bigwedge }\tau _{P}(G_{xy}^{c})\geq \lambda ^{\prime
}.$ Since $G_{x}\supseteq B_{x},$ $N_{x}^{P^{X}}(G_{x})\geq
N_{x}^{P^{X}}(B_{x})>\lambda ^{\prime },$ i.e., $\underset{x\in H\subseteq
G_{x}}{\bigvee }\tau _{P}(H)>\lambda ^{\prime }.$ Thus there exists $H_{x}$
such that $x\in H_{x}\subseteq G_{x}$ and $\tau _{P}(H_{x})>\lambda ^{\prime
}.$ Hence for any $x\in A$, there exists $H_{x}$ and $G_{x}$ such that $x\in
H_{x}\subseteq G_{x}\subseteq U,$ $\tau _{P}(H_{x})>\lambda ^{\prime }$ and $%
W\supseteq \underset{x\in A}{\bigcup }G_{x}\supseteq \underset{x\in A}{%
\bigcup }H_{x}\supseteq A.$ We define $\Re \in \Im (P(A))$ as follows:%


$${\Re (D)}=\left\{
\begin{array}{lc}
\underset{H_{x}\cap A=D}{\bigvee }\tau _{P}(H_{x}), & {\rm there\ exists }%
H_{x}\ {\rm such \ that } H_{x}\cap A=D, \\
0, & \ {\rm otherwise}
\end{array}
\right. $$

Let $\Gamma _{P}(A,\tau /A)=\mu >\mu -\epsilon ($ $\epsilon >0).$ Then $%
1-K_{P}(\Re ,A)+\underset{\wp \leq \Re }{\bigvee }[K(\Re ,A)\otimes FF(\wp
)]>\mu -\epsilon ,$ where $$[K(\Re ,A)]=\underset{x\in A}{\bigwedge }\
\underset{x\in B}{\bigvee }\Re (B)=\underset{x\in A}{\bigwedge }\ \underset{%
x\in D}{\bigvee }\Re (D)=\underset{x\in A}{\bigwedge }\ \underset{x\in D}{%
\bigvee }\ \underset{H_{x^{\prime }}\cap A=D}{\bigvee }\tau
_{P}(H_{x^{\prime }})\geq \lambda ^{\prime }$$

and
$[\Re \subseteq \tau _{P}\backslash A]=\underset{B\subseteq X}{\bigwedge }%
\min (1,1-\Re (B)+\tau _{P}\backslash A(B))$ \newline
\ \ \ \ \ \ \ \ \ \ \ \ \ \ \ \ \ \ \ \ $=\underset{B\subseteq X}{\bigwedge }%
\min (1,1-\underset{H_{x}\cap A=B}{\bigvee }\tau _{P}(H_{x})+\underset{H\cap
A=B}{\bigvee }\tau _{P}(H))=1.$ \newline
So, $K_{P}(\Re ,A)=[K(\Re ,A)]\geq \lambda ^{\prime }$. By (*), \newline
$[K(\Re ,A)\otimes FF(\wp )]>\mu -\epsilon -1+K_{P}(\Re ,A)\geq \mu
-\epsilon -1+\lambda ^{\prime }\geq $ $\lambda -\epsilon .$ \newline
Thus $\underset{x\in A}{\bigwedge }\ \underset{x\in E}{\bigvee }\Re
(E)+1-\bigwedge \{\delta :F(\wp _{\delta })\}-1>\lambda -\epsilon ,$ and $%
\underset{x\in A}{\bigwedge }\ \underset{x\in E}{\bigvee }\Re (E)>\lambda
-\epsilon +\bigwedge \{\delta :F(\wp _{\delta })\}.$ \newline
Hence there exists $\beta >0$ such that $F(\wp _{\beta })$ and $\underset{%
x\in A}{\bigwedge }\ \underset{x\in D}{\bigvee }\Re (D)>\lambda -\epsilon
+\beta .$ Therefore for any $x\in A,$ there exists $D_{x}\subseteq A$ such
that $\wp (D_{x})>\lambda -\epsilon +\beta $ and $\underset{x\in A}{\bigcup }%
D_{x}\supseteq A.$ Suitably choose $\epsilon $ such that $\lambda -\epsilon
>0,$ then $\wp (D_{x})>\beta >0.$ Since $\Re (D_{x})\geq \wp (D_{x})>0,$ $%
D_{x}=H_{x^{\prime }}\cap A,$ i.e., $H_{x^{\prime }}\cap A\in \wp _{\beta }.$
By $F(\wp _{\beta })$, so there exists finite $H_{x_{1}^{\prime
}},H_{x_{2}^{\prime }},...,H_{x_{n}^{\prime }}$ such that $%
\bigcup_{i=1}^{n}H_{x_{i}^{\prime }}\supseteq A$ and $%
\bigcup_{i=1}^{n}H_{x_{i}^{\prime }}\subseteq
\bigcup_{i=1}^{n}G_{x_{i}^{\prime }}.$Set $V=\bigcup_{i=1}^{n}G_{x_{i}^{%
\prime }},$ and $V^{c}=\bigcap_{i=1}^{n}G_{x_{i}^{\prime }}^{c},$ $%
A\subseteq V\subseteq U,$ and $\tau _{P}(V^{c})\geq \underset{1\leq i\leq n}{%
\bigwedge }\tau _{P}(G_{x_{i}^{\prime }}^{c})\geq \lambda ^{\prime }>\lambda
.$ Since for any $x\in A,$ $G_{x}\subseteq W\subseteq C\cap U\subseteq C,$
we have $V=\bigcup_{i=1}^{n}G_{x_{i}^{\prime }}\subseteq W\subseteq C.$
Because $\tau _{P}\backslash C(C-V)=\underset{D\cap C=C\cap V^{c}}{\bigvee }%
\tau _{P}(D)\geq \tau _{P}(V^{c})\geq \lambda ^{\prime }.$ Thus by (*), $%
\tau _{P}\backslash C(C-V)+\Gamma _{P}(C,\tau /C)-1>\lambda .$By Theorem 5.1
in [21], $\Gamma _{P}(V,\tau /V)=\Gamma _{P}(V,\tau /C/V)\geq \lbrack \Gamma
_{P}(C,\tau /C)\otimes \tau _{P}\backslash C(C-V)]>\lambda .$\newline

Finally, we have for any $x\in A,$\newline
$N_{x}^{P^{X}}(V)=N_{x}^{P^{X}}(\bigcup_{i=1}^{n}G_{x_{i}^{\prime }})\geq
N_{x}^{P^{X}}(\bigcup_{i=1}^{n}H_{x_{i}^{\prime }})\geq $ $\tau
_{P}(\bigcup_{i=1}^{n}H_{x_{i}^{\prime }})\geq $ $\underset{1\leq i\leq n}{%
\bigwedge }\tau _{P}(H_{x_{i}^{\prime }})\geq $ $\lambda ^{\prime }>\lambda
. $ So $N_{A}^{P^{X}}(V)=\underset{x\in A}{\bigwedge }N_{x}^{P^{X}}(V)\geq
\lambda .$ Therefore $N_{A}^{P^{X}}(V)\wedge \tau _{P}(V^{c})\wedge \Gamma
_{P}(V,\tau /V)\geq \lambda .$ \newline
Thus $\underset{V\subseteq U}{\bigvee }(N_{A}^{P^{X}}(V)\wedge \tau
_{P}(V^{c})\wedge \Gamma _{P}(V,\tau /V))\geq \lambda .$
\end{proof}
\begin{theorem}
Let $(X,\tau )$ and $(Y,\sigma )$\ be two fuzzifying topological space and $%
f\in Y^{X}$ be surjective. Then $\vDash L_{P}C(X,\tau )\otimes
C_{P}(f)\otimes O(f)\rightarrow LC(Y,\sigma ).$ For the definition of O(f),
see [25].
\end{theorem}
\begin{proof}
If $[L_{P}C(X,\tau )\otimes C_{P}(f)\otimes O(f)]>\lambda >0,$ then for any $%
x\in X,$there exists $U\subseteq X,$ such that $[N_{x}^{P^{X}}(U)\otimes
\Gamma _{P}(U,\tau /U)\otimes C_{P}(f)\otimes O(f)]>\lambda .$ Since $%
N_{x}^{P^{X}}(U)=\underset{x\in V\subseteq U}{\bigvee }\tau _{P}(V),$ so
there exists $V^{\prime }\subseteq X$ such that $x\in V^{\prime }\subseteq U$
and $[\tau _{P}(V^{\prime })\otimes \Gamma _{P}(U,\tau /U)\otimes
C_{P}(f)\otimes O(f)]>\lambda .$ By Theorem 8 in [1] , $[\Gamma _{P}(U,\tau
/U)\otimes C_{P}(f)]\leq \lbrack \Gamma (f(U),\sigma /f(U))]$ and \newline
$[\tau (V^{\prime })\otimes O(f)]=\max (0,\tau (V^{\prime })+O(f)-1)$

\ \ \ \ \ \ \ \ \ \ \ \ \ \ \ \ \ $=\max (0,\tau (V^{\prime })+\underset{%
V\subseteq X}{\bigwedge }\min (1,1-\tau (V^{\prime })+\sigma (f(V)))-1)$

\ \ \ \ \ \ \ \ \ \ \ \ \ \ \ \ \ $\leq \max (0,\tau (V^{\prime })+1-\tau
(V^{\prime })+\sigma (f(V))-1)$

\ \ \ \ \ \ \ \ \ \ \ \ \ \ \ \ \ $=\sigma (f(V))\leq
N_{f(x)}^{^{Y}}(f(V^{\prime }))\leq N_{f(x)}^{^{Y}}(f(U)).$ \newline
Since $f$ is surjective, \newline
$LC(Y,\sigma )=LC(f(X),\sigma )=\underset{y\in f(x)\subseteq f(X)}{\bigwedge
}\ \underset{U^{\prime }=f(U)\subseteq f(X)}{\bigvee }[N_{y}^{^{Y}}(U^{%
\prime })\otimes \lbrack \Gamma (U^{\prime },\sigma /U^{\prime })]$

\ \ \ \ \ \ \ \ \ \ \ $\geq \underset{y\in f(x)\subseteq f(X)}{\bigwedge }\
[N_{f(x)}^{^{Y}}(f(U))\otimes \lbrack \Gamma (f(U),\sigma /f(U))]$

\ \ \ \ \ \ \ \ \ \ \ $\geq \underset{y\in f(x)\subseteq f(X)}{\bigwedge }\
[\tau (V^{\prime })\otimes O(f)\otimes \Gamma _{P}(U,\tau /U)\otimes
C_{P}(f)]\geq \lambda .$
\end{proof}

\begin{theorem}
Let $(X,\tau )$ and $(Y,\sigma )$\ be two fuzzifying topological space and $%
f\in Y^{X}$ be surjective. Then $\vDash L_{P}C(X,\tau )\otimes
I_{P}(f)\otimes O_{P}(f)\rightarrow L_{P}C(Y,\sigma ).$
\end{theorem}

\begin{proof}
By Theorem 9 in [1], the proof is similar to the proof of Theorem 4.6.
\end{proof}

Theorems 4.6 and 4.7 are a generalization of the following corollary.

\begin{corollary}
Let $(X,\tau )$ and $(Y,\sigma )$\ be two topological space and $f:(X,\tau
)\rightarrow (Y,\sigma )$ be surjective mapping. If f is a pre-continuous
(resp. pre-irresolute), open (resp. pre-open) and X is locally strong
compact, then Y is locally compact (resp. locally strong compact) space.
\end{corollary}

\begin{theorem}
Let $\{(X_{s},\tau _{s}):s\in S\}$ be a family of fuzzifying topological
spaces, then \newline
$\vDash L_{P}C(\underset{s\in S}{\prod }X_{s},\underset{s\in S}{\prod }(\tau
_{P})_{s})\rightarrow \forall s(s\in S\wedge L_{P}C(X_{s},(\tau
_{P})_{s})\wedge $ $\exists T(T\Subset S\wedge \forall t(t\in S-T\wedge
\Gamma _{P}(X_{t},\tau _{t}))).$
\end{theorem}

\begin{proof}
It suffices to show that \begin{eqnarray*}
L_{P}C(\underset{s\in S}{\prod }X_{s},\underset{s\in S}{\prod }(\tau
_{P})_{s})\leq \underset{s\in S}{\bigwedge }[L_{P}C(X_{s},(\tau
_{P})_{s})\wedge \underset{T\Subset S}{\bigvee }\ \underset{t\in S-T}{%
\bigwedge }\Gamma _{P}(X_{t},\tau _{t})]. \end{eqnarray*}
From Theorem 4.7 and Lemma 3.1 we have for any $t\in S,$ \newline
$$L_{P}C(\underset{s\in S}{\prod }X_{s},\underset{s\in S}{\prod }(\tau
_{P})_{s})=[L_{P}C(\underset{s\in S}{\prod }X_{s},\underset{s\in S}{\prod }%
(\tau _{P})_{s})\otimes C_{P}(p_{t})\otimes O_{P}(p_{t})]\leq
L_{P}C(X_{t},\tau _{t}).$$
So, $$\ \underset{t\in S-T}{\bigwedge }L_{P}C(X_{t},\tau _{t})\geq L_{P}C(%
\underset{s\in S}{\prod }X_{s},\underset{s\in S}{\prod }(\tau _{P})_{s}).$$

By Theorem 3.2 we have
\begin{eqnarray*}\underset{T\Subset S}{\bigvee }\ \underset{t\in S-T}{\bigwedge }\Gamma
_{P}(X_{t},\tau _{t})&\geq& \lbrack \underset{U\subseteq \underset{s\in S}{%
\prod }X_{s}}{\bigvee }\ \Gamma _{P}(U,\underset{s\in S}{\prod }(\tau
_{P})_{s}\backslash U)\otimes \underset{X\subseteq \underset{s\in S}{\prod }%
X_{s}}{\bigvee }N_{x}^{P^{X}}(U))]\\ &\geq& \underset{U\subseteq
\underset{s\in S}{\prod }X_{s}}{\bigvee }\ \underset{X\subseteq \underset{%
s\in S}{\prod }X_{s}}{\bigvee }[\Gamma _{P}(U,\underset{s\in S}{\prod }(\tau
_{P})_{s}\backslash U)\otimes N_{x}^{P^{X}}(U))]\\ &\geq& \underset{X\subseteq
\underset{s\in S}{\prod }X_{s}}{\bigwedge }\ \underset{U\subseteq \underset{%
s\in S}{\prod }X_{s}}{\bigvee }[\Gamma _{P}(U,\underset{s\in S}{\prod }(\tau
_{P})_{s}\backslash U)\otimes N_{x}^{P^{X}}(U))]\\ &=&L_{P}C(\underset{s\in S}{\prod
}X_{s},\underset{s\in S}{\prod }(\tau _{P})_{s}). \end{eqnarray*}
Therefore $$L_{P}C(\underset{s\in S}{\prod }X_{s},\underset{s\in S}{\prod }%
(\tau _{P})_{s})\leq \lbrack \underset{t\in S-T}{\bigwedge }%
L_{P}C(X_{t},\tau _{t})\wedge \underset{T\Subset S}{\bigvee }\ \underset{%
t\in S-T}{\bigwedge }\Gamma _{P}(X_{t},\tau _{t})].$$
\end{proof}

We can obtain the following corollary in crisp setting.

\begin{corollary}
Let $\{X_{\lambda }:\lambda \in \Lambda \}$ be a family of nonempty
topological spaces. If $\underset{\lambda \in \Lambda }{\prod }X_{\lambda }$
is locally strong compact, then each $X_{\lambda }$ is locally strong
compact \ and all but finitely many $X_{\lambda }$ are strong compact
\end{corollary}

\textbf{Conclusion}: The main contributions of the present work is to
give characterizations of fuzzifying strong compactness. We also define the
concept of locally strong compactness of fuzzifying topological spaces and
study some basic properties of such spaces. We also state some open problems for
future study: \begin{enumerate}
\item[(1)] Is it possible to generalize the results in the present work to lattice-valued logic.
\item[(2)] What is the justification for fuzzifying locally strong compactness in
the setting of $(2, L)$ topologies. In fact in $%
(M,L)$-topologies for more general setting.
\item[(3)] Further, the fuzzifying topological spaces in [18] form a fuzzy category. Perhaps, this will become a motivation for further study of the fuzzy
category.

\end{enumerate}

\end{document}